\documentclass[11pt,oneside,reqno]{amsart}
\usepackage[pagewise]{lineno}
\usepackage{graphicx}
\usepackage{mathptmx}

\usepackage{amscd}
\usepackage{amsthm}
\usepackage{amsxtra}
\usepackage{latexsym}
\usepackage{amssymb}
\usepackage{amsfonts}
\usepackage{amsmath}
\usepackage{amsrefs}
\usepackage{mathrsfs}
\usepackage{centernot}

\usepackage{upref}
\usepackage{txfonts}

\usepackage{color,soul}
\usepackage{xcolor}

\usepackage[colorinlistoftodos]{todonotes}

\usepackage[bookmarksnumbered, colorlinks, plainpages]{hyperref}
\hypersetup{colorlinks=true,linkcolor=red, anchorcolor=green, citecolor=cyan, urlcolor=red, filecolor=magenta, pdftoolbar=true}

\usepackage[left=2cm, right=2cm, top=2cm, bottom=2cm]{geometry}

\allowdisplaybreaks

\theoremstyle{plain}
\newtheorem{thm}{Theorem}[section]

\newtheorem{cor}[thm]{Corollary}

\theoremstyle{definition}
\newtheorem{rem}[thm]{Remark}

\newtheorem{defi}[thm]{Definition}

\newtheorem{conv}[thm]{Convention}

\numberwithin{thm}{section}
\numberwithin{equation}{section}

\def\esup{\operatornamewithlimits{ess\,sup}}

\def\ces{\operatorname{Ces}}
\def\cop{\operatorname{Cop}}
\def\Ces{\operatorname{ces}}

\def\qq{\qquad}
\def\rw{\rightarrow}

\def\dn{\downarrow}

\def\ls{\lesssim}

\def\M{\mathcal M}

\def\la{\lambda}
\def\La{\Lambda}

\def\i{\infty}
\def\I{(0,\i)}

\def\R{\mathbb R}

\def\R{\mathbb R}

\def\M{\mathfrak M}

\def\W{{\mathcal W}}

\def\mp{{\mathfrak M}}

\def\rn{\R^n}

\def\la{\lambda}
\def\La{\Lambda}

\def\i{\infty}
\def\I{(0,\i)}

\begin{document}

\title[]{Boundedness of weighted iterated Hardy-type operators involving suprema from weighted Lebesgue spaces into weighted Ces\`{a}ro function spaces}

\author[]{R.Ch.~Mustafayev and N. B\.{I}LG\.{I}\c{C}L\.{I}} 

\address{Rza Mustafayev, Department of Mathematics, Faculty of Science, Karamanoglu Mehmetbey University, Karaman, 70100, Turkey}
\email{rzamustafayev@gmail.com}

\address{Nevin Bilgi\c{c}li, Department of Mathematics, Faculty of Science and Arts, Kirikkale 	University, 71450 Yahsihan, Kirikkale, Turkey}
\email{nevinbilgicli@gmail.com}



\begin{abstract}
    In this paper the boundedness of the weighted iterated Hardy-type operators $T_{u,b}$ and $T_{u,b}^*$ involving suprema from  weighted Lebesgue space $L_p(v)$ into weighted Ces\`{a}ro function spaces ${\operatorname{Ces}}_{q}(w,a)$ are characterized. These results allow us to obtain the characterization of the boundedness of the supremal operator $R_u$ from $L^p(v)$ into ${\operatorname{Ces}}_{q}(w,a)$ on the cone of monotone non-increasing functions. For the convenience of the reader, we formulate the statement on the boundedness of the weighted Hardy operator $P_{u,b }$ from $L^p(v)$ into ${\operatorname{Ces}}_{q}(w,a)$ on the cone of monotone non-increasing functions. Under additional condition on $u$ and $b$, we are able to characterize the boundedness of weighted iterated Hardy-type operator $T_{u,b}$ involving suprema from $L^p(v)$ into ${\operatorname{Ces}}_q(w,a)$ on the cone of monotone non-increasing functions. At the end of the paper, as an application of obtained results, we calculate the norm of the fractional maximal function $M_{\gamma}$ from $\Lambda^p(v)$ into $\Gamma^q(w)$. 	
\end{abstract}

\date{}
\maketitle


\section{Introduction}\label{introduction}

Many Banach spaces which play an important role in functional
analysis and its applications are obtained in a special way: the
norms of these spaces are generated by positive sublinear operators
and by $L_p$-norms.

In connection with Hardy and Copson operators
$$
(Pf)(x) : = \frac{1}{x} \int_0^x f(t)\,dt \qq \mbox{and} \qq (Qf)(x)
: = \int_x^{\infty} \frac{f(t)}{t}\,dt,\qq (x > 0),
$$
the classical Ces\`{a}ro function space
$$
\ces(p) : = \bigg\{ f:\, \|f\|_{\ces(p)} : =  \bigg( \int_0^{\infty}
\bigg( \frac{1}{x} \int_0^x |f(t)|\,dt \bigg)^p\,dx
\bigg)^{\frac{1}{p}} < \infty \bigg\},
$$
and the classical Copson function space
$$
\cop(p) : = \bigg\{ f:\, \|f\|_{\cop(p)} : = \bigg( \int_0^{\infty}
\bigg( \int_x^{\infty} \frac{|f(t)|}{t}\,dt \bigg)^p\,dx
\bigg)^{\frac{1}{p}} < \infty \bigg\},
$$
where $1 < p \le \infty$, with the usual modifications if $p =
\infty$, are of interest.

The classical Ces\`{a}ro function spaces $\ces(p)$ have been
introduced in 1970 by Shiue \cite{shiue}. These spaces have been defined analogously to the Ces\`{a}ro sequence spaces that appeared two years earlier in
\cite{prog} when the Dutch Mathematical Society posted a problem to find a representation of their dual spaces. In 1971 Leibowitz proved that $\Ces_1 = \{0\}$ and for $1 < q < p \leq \infty$, $\ell_p$ and $\Ces_q$ sequence spaces are proper subspaces of $\Ces_p$ \cite{Leibowitz}. The problem posted \cite{prog} was resolved by Jagers \cite{jagers} in 1974 who gave an explicit
isometric description of the dual of Ces\`{a}ro sequence space. In \cite{syzhanglee}, Sy, Zhang and Lee gave a description of dual spaces of $\ces(p)$ spaces based on Jagers' result. In 1996 different, isomorphic description due to Bennett appeared in \cite{bennett1996}.  In \cite[Theorem 21.1]{bennett1996} Bennett observes that the classical Ces\`{a}ro function space and the classical Copson function space coincide for $p > 1$. He also derives estimates for the norms of the corresponding inclusion operators. The same result, with different estimates, is due to Boas \cite{boas1970}, who in fact obtained the integral analogue of the Askey-Boas Theorem \cite[Lemma 6.18]{boas1967} and \cite{askeyboas}. These results generalized in \cite{grosse} using the blocking technique. In \cite{astasmal2009} they investigated dual spaces for $\ces (p)$ for $1 < p < \infty$. Their description can be viewed as being analogous to one given for sequence spaces in \cite{bennett1996}. For a long time, Ces\`{a}ro function spaces have not attracted a lot of attention contrary to their sequence
counterparts. In fact there is quite rich literature concerning different topics studied in Ces\`{a}ro sequence spaces as for instance in \cites{CuiPluc,cuihud1999,cuihud2001,chencuihudsims,cuihudli}.
However, recently in a series of papers, Astashkin and Maligranda started to study the structure of Ces\`{a}ro function spaces. Among others, in \cite{astasmal2009} they investigated dual spaces for $\ces (p)$ for $1 < p < \infty$. Their description can be viewed as being analogous to one given for sequence spaces in \cite{bennett1996} (For more detailed information about history of classical Ces\`{a}ro spaces see recent survey paper \cite{asmalsurvey}). 

Throughout the paper we assume that $I : = (a,b)\subseteq (0,\i)$.
By $\mp (I)$ we denote the set of all measurable functions on $I$.
The symbol $\mp^+ (I)$ stands for the collection of all $f\in\mp
(I)$ which are non-negative on $I$, while $\mp^{+,\dn}(I)$ is used to denote the subset of those functions which are
non-increasing on $I$, respectively. A weight is a function $v \in {\mathfrak M}^+(0,\infty)$ such that $0 < V(x) < \infty$ for all $x \in (0,\infty)$,
where 
$$
V(x) : = \int_0^x v(t)\,dt.
$$
The family of all weight functions (also called just weights) on $(0,\infty)$ is given by $\W(0,\infty)$.

For $p\in (0,\i]$ and $w\in \mp^+(I)$, we define the functional
$\|\cdot\|_{p,w,I}$ on $\mp (I)$ by
\begin{equation*}
\|f\|_{p,w,I} : = \left\{\begin{array}{cl}
\left(\int_I |f(x)|^p w(x)\,dx \right)^{1/p} & \qq\mbox{if}\qq p<\i \\
\esup_{I} |f(x)|w(x) & \qq\mbox{if}\qq p=\i.
\end{array}
\right.
\end{equation*}

If, in addition, $w\in \W(I)$, then the weighted Lebesgue space
$L^p(w,I)$ is given by
\begin{equation*}
L^p(w,I) = \{f\in \mp (I):\,\, \|f\|_{p,w,I} < \i\}
\end{equation*}
and it is equipped with the quasi-norm $\|\cdot\|_{p,w,I}$.

When $I=(0,\infty)$, we write $L^p(w)$ instead of $L^p(w,(0,\infty))$.

We adopt the following usual conventions.
\begin{conv}\label{Notat.and.prelim.conv.1.1}
	We adopt the following conventions:
	
	\begin{itemize}
		\item Throughout the paper we put $0 \cdot \infty = 0$, $\infty / \infty =
		0$ and $0/0 = 0$.
		
		\item If $p\in [1,+\infty]$, we define $p'$ by $1/p + 1/p' = 1$.
		
		\item If $0 < q < p < \infty$, we define $r$ by $1 / r  = 1 / q - 1 / p$.
		
		\item If $I = (a,b) \subseteq {\mathbb R}$ and $g$ is monotone
		function on $I$, then by $g(a)$ and $g(b)$ we mean the limits
		$\lim_{x\rightarrow a+}g(x)$ and $\lim_{x\rightarrow b-}g(x)$, respectively.
	\end{itemize}
\end{conv}

Throughout the paper, we always denote by $c$ and $C$ a positive
constant, which is independent of main parameters but it may vary
from line to line. However a constant with subscript or superscript
such as $c_1$ does not change in different occurrences. By
$a\lesssim b$, ($b\gtrsim a$) we mean that $a\leq \la b$, where
$\la>0$ depends on inessential parameters. If $a\lesssim b$ and
$b\lesssim a$, we write $a\approx b$ and say that $a$ and $b$ are
equivalent. 

Unless a special remark is made, the differential element $dx$ is omitted when the integrals under consideration are the Lebesgue integrals.

The weighted Ces\`{a}ro and Copson function spaces are defined as
follows:
\begin{defi}\label{defi.2.0}
	Let  $0 <p \le \infty$, $u \in \mp^+ \I$ and $v \in \W (0,\infty)$. The
	weighted Ces\`{a}ro and Copson spaces are defined by
	\begin{align*}
	\ces_{p} (u,v) : & = \bigg\{ f \in \mp^+ \I: \|f\|_{\ces_{p}
		(u,v)} : = \big\| \|f\|_{1,v,(0,\cdot)}
	\big\|_{p,u,\I} < \i \bigg\}, \\
	\intertext{and} \cop_{p} (u,v) : & = \bigg\{ f \in \mp^+ \I:
	\|f\|_{\cop_{p} (u,v)} : = \big\| \|f\|_{1,v,(\cdot,\i)}
	\big\|_{p,u,\I} < \i \bigg\},
	\end{align*}
	respectively.
	
	When $v \equiv 1$ on $(0,\infty)$, we simply write $\ces_{p} (u)$ and $\cop_{p} (u)$ instead of $\ces_{p} (u,v)$ and $\cop_{p} (u,v)$, respectively.
\end{defi}

Recall that $\ces_{p} (u,v)$ and $\cop_{p} (u,v)$ are contained in the scale of weighted Ces\`{a}ro and Copson function spaces $\ces_{p,q} (u,v)$ and $\cop_{p,q} (u,v)$ defined in \cite{gmu_2017}. Obviously, $\ces(p) = \ces_p
(x^{-1})$ and $\cop(p) = \cop_p (x^{-1})$. In \cite{kamkub},
Kami{\'n}ska and Kubiak computed the dual norm of the Ces\`{a}ro
function space $\ces_{p}(u)$, generated by $1 < p < \infty$ and an
arbitrary positive weight $u$. A description presented in
\cite{kamkub} resembles the approach of Jagers \cite{jagers} for
sequence spaces.

Let $u \in \W\I \cap C\I$, $b \in \W\I$ and $B(t) : = \int_0^t
b(s)\,ds$. Assume that $b$ is a weight such that $b(t) > 0$ for a.e. $t \in (0,\infty)$. The weighted iterated Hardy-type operators involving suprema
$T_{u,b}$ and $T_{u,b}^*$ are defined at $g \in \M^+ \I$ by
\begin{align*}
(T_{u,b} g)(t) & : = \sup_{t \le \tau}
\frac{u(\tau)}{B(\tau)} \int_0^{\tau} g(y)b(y)\,dy,\qquad t \in \I, \\
(T_{u,b}^* g)(t) & : = \sup_{t \le \tau}
\frac{u(\tau)}{B(\tau)} \int_{\tau}^{\infty} g(y)b(y)\,dy,\qquad t \in \I.
\end{align*}

Such operators have been found indispensable in the search for
optimal pairs of rearrangement-invariant norms for which a
Sobolev-type inequality holds (cf. \cite{kerp}). They constitute a
very useful tool for characterization of the associate norm of an
operator-induced norm, which naturally appears as an optimal domain
norm in a Sobolev embedding (cf. \cite{pick2000}, \cite{pick2002}).
Supremum operators are also very useful in limiting interpolation
theory as can be seen from their appearance for example in
\cite{evop}, \cite{dok}, \cite{cwikpys}, \cite{pys}. Recall that $T_{u,b}$ successfully controls non-increasing rearrangements of wide range of maximal functions (see, for instance, \cite{musbil} and references therein).

It was shown in  \cite{gop} that for every $h \in \mp^+(0,\infty)$ and $t \in (0,\infty)$
$$
(T_{u,b} h)(t) = (T_{\bar{u},b} h) (t),
$$
where
$$
\bar{u} (t) : = B(t) \sup_{t \le \tau} \frac{u(\tau)}{B(\tau)}, \qquad t\in (0,\infty).
$$
Moreover, if the condition 
\begin{equation}\label{add.cond.}
\sup_{0 < t < \infty} \frac{u(t)}{B(t)} \int_0^t
\frac{b(\tau)}{u(\tau)}\,d\tau < \infty.
\end{equation}
holds, then for all $f \in \mp^{+,\dn} (0,\infty)$,
\begin{equation}\label{Split}
(T_{u,b}f)(t) \approx (R_u f)(t) + (P_{\bar{u},b}f)(t), \qquad t\in (0,\infty),
\end{equation}
where the supremal operator $R_u$ and the weighted Hardy operator $P_{u,b}$ are defined for $h \in \mp^+(0,\infty)$ and $t \in (0,\infty)$ by
\begin{align*}
(R_u h) (t) : & = \sup_{t \le \tau} u(\tau) h(\tau),  \\
(P_{u,b} h)(t) : & = \frac{u(t)}{B(t)} \int_0^t h(\tau) b(\tau) \,d \tau, 
\end{align*}
respectively.

Recall that the boundedness of $R_u$ from $L^p(v)$ into $L^q(w)$ on the cone of monotone non-increasing functions, that is,
the validity of the inequality
\begin{equation}\label{eq.R}
\|R_u f\|_{L^q(w)} \le C \, \|f\|_{L^p(v)}, \qq f \in \mp^{+,\dn} (0,\infty)
\end{equation}
was completely characterized in \cite {gop} in the case $0 < p \le q < \infty$. In the case $0 < q < p < \infty$,  \cite {gop} provides solution when $u$ is equivalent to a non-decreasing function on $(0,\infty)$. The complete solution of inequality  \eqref{eq.R} using a certain reduction method was presented in \cite{GogMusISI}. Another solution of \eqref{eq.R} was obtained in \cite{krep}.

Note that inequality
\begin{equation}\label{Reduction.Theorem.Thm.2.5}
\| P_{u,b} (f) \|_{q,w,(0,\infty)} \le c \| f \|_{p,v,(0,\infty)}, \qq f \in \mp^{+,\dn} (0,\infty)
\end{equation}
was considered by many authors and there exist several characterizations of this inequality (see, papers \cites{cpss,bengros,gjop,cgmp2008,GogStep,GogMusIHI}).

The complete characterizations of inequality  
\begin{equation}\label{Tub.thm.1.eq.1}
\|T_{u,b}f \|_{q,w,\I} \le C \| f \|_{p,v,\I}, \qq f \in
\mp^{+,\dn}(0,\infty)
\end{equation}
for $0 < q \le \infty$, $0 < p \le \infty$ were given in \cite{GogMusISI} and \cite{musbil}. Inequality \eqref{Tub.thm.1.eq.1} was characterized in \cite[Theorem
3.5]{gop} under condition \eqref{add.cond.}. Note that the case when $0 < p \le 1 < q < \infty$ was not
considered in \cite{gop}. It is also worth to mention that in the
case when $1 < p < \infty$, $0 < q < p < \infty$, $q \neq 1$
\cite[Theorem 3.5]{gop} contains only discrete condition. In
\cite{gogpick2007} the new reduction theorem was obtained when $0 <
p \le 1$, and this technique allowed to characterize inequality
\eqref{Tub.thm.1.eq.1} when $b \equiv 1$, and in the case when $0 <
q< p \le 1$, \cite{gogpick2007} contains only discrete condition. Using the results in  \cites{PS_Proc_2013,PS_Dokl_2013,PS_Dokl_2014,P_Dokl_2015}, another characterization of  \eqref{Tub.thm.1.eq.1}  was obtained  in  \cite{StepSham} and \cite{Sham}. 

In this paper we investigate the boundedness of $T_{u,b}$ and $T_{u,b}^*$ from the weighted Lebesgue spaces $L_p(v)$ into the weighted Ces\`{a}ro spaces $\ces_{q} (w,a)$, when $1 < p,\, q < \infty$ (see, Theorems \ref{aux.thm.1} and \ref{aux.thm.2}). These results allow us to obtain the characterization of the boundedness of $R_u$ from $L^p(v)$ into $\ces_{q}(w,a)$ on the cone of monotone non-increasing functions (see, Theorem \ref{thm.R}). For the convenience of the reader, we formulate the statement on the boundedness of $P_{u,b }$ from $L^p(v)$ into $\ces_{q}(w,a)$ on the cone of monotone non-increasing functions (see, Theorem \ref{aux.thm.3}). In view of \eqref{Split}, we are able to characterize the boundedness of $T_{u,b}$ from $L^p(v)$ into $\ces_q(w,a)$ on the cone of monotone non-increasing functions (see, Theorem \ref{thm.T}). At the end of the paper, as an application of obtained results, we calculate the norm of the fractional maximal function $M_{\gamma}$ from $\Lambda^p(v)$ into $\Gamma^q(w)$.

The paper is organized as follows. We start with formulations of "an integration by parts" formula in Section~\ref{integration by parts}. The boundedness results for $T_{u,b}$ and $T_{u,b}^*$ from $L^p(v)$ into $\ces_q(w,a)$ are presented in Section \ref{main results}. The characterizations of the boundedness of $R_u$, $P_{u,b}$ and $T_{u,b}$ from $L^p(v)$ into $\ces_{q}(w,a)$ on the cone of monotone non-increasing functions are given in Sections \ref{R}, \ref{P} and \ref{T}, respectively. Finally, the obtained in previous sections results  are applied to calculate the norm of the operator $M_{\gamma}: \Lambda_p(v) \rw \Gamma_q(w)$ in Section \ref{Appl.}.



\

\section{"An integration by parts" formula}\label{integration by parts}

\


We recall the following "an integration by parts" formula. For the convenience of the reader we give the proof here (cf. \cite[Lemma, p. 176]{step_1993}).
\begin{thm}\label{thm.IBP.0}
	Let $\alpha > 0$. Let $g$ be a non-negative function on $(0,\infty)$ such that $0 < \int_0^t g < \infty$, $t > 0$ and let $f$ be a non-negative non-increasing right-continuous function on $(0,\infty)$. Then
	\begin{align*}
	A_1 : = \int_0^{\infty} \bigg( \int_0^t g \bigg)^{\alpha} g(t) [f(t) - \lim_{t \rw +\infty} f(t)]\,dt < \infty \quad \Longleftrightarrow \quad A_2 : = \int_{(0,\infty)} \bigg( \int_0^t g \bigg)^{\alpha + 1}\,d[-f(t)] < \infty.
	\end{align*}
	Moreover, $A_1 \approx A_2$.
\end{thm}

\begin{proof}
	Assume at first that $\lim_{t \rw +\infty} f(t) = 0$. Let 
	$$
	A_1 = \int_0^{\infty} \bigg( \int_0^t g \bigg)^{\alpha} g(t) f(t)\,dt < \infty.
	$$
	Then
	$$
	\int_0^x \bigg( \int_0^t g \bigg)^{\alpha} g(t) f(t)\,dt \rightarrow 0, \quad \mbox{as} \quad x \rightarrow 0+.
	$$
	Since
	\begin{align*}
	\int_0^x \bigg( \int_0^t g \bigg)^{\alpha} g(t) f(t)\,dt \ge f(x) \, \int_0^x \bigg( \int_0^t g \bigg)^{\alpha} g(t) \,dt \approx f(x) \, \bigg( \int_0^x g \bigg)^{\alpha + 1},  \quad x > 0, 
	\end{align*}
	we have that
	$$
	f(x) \, \bigg( \int_0^x g \bigg)^{\alpha + 1} \rightarrow 0, \quad \mbox{as} \quad x \rightarrow 0+.
	$$
	Integrating by parts, we get that
	\begin{align*}
	A_2 & = \int_{(0,\infty)} \bigg( \int_0^t g \bigg)^{\alpha + 1}\,d[-f(t)] = - f(t) \, \bigg( \int_0^t g \bigg)^{\alpha + 1} \bigg|_{0}^{\infty} + 
	\int_{(0,\infty)} f(t)\,d\bigg( \int_0^t g \bigg)^{\alpha + 1} \\
	& = \lim_{t \rw 0+} f(t) \, \bigg( \int_0^t g \bigg)^{\alpha + 1} - \lim_{t \rw +\infty} f(t) \, \bigg( \int_0^t g \bigg)^{\alpha + 1} + (\alpha + 1) \int_0^{\infty} \bigg( \int_0^t g \bigg)^{\alpha} g(t) f(t)\,dt \\
	& \le (\alpha + 1) \int_0^{\infty} \bigg( \int_0^t g \bigg)^{\alpha} g(t) f(t)\,dt = (\alpha + 1) A_1.
	\end{align*}
	Thus
	$$
	A_2 \lesssim A_1.
	$$
	
	Now assume that 	
	$$
	A_2 : = \int_{(0,\infty)} \bigg( \int_0^t g \bigg)^{\alpha + 1}\,d[-f(t)] < \infty.
	$$
	Then
	$$
	\int_{[x,\infty)} \bigg( \int_0^t g \bigg)^{\alpha + 1}\,d[-f(t)] \rightarrow 0, \quad \mbox{as} \quad x \rightarrow +\infty.
	$$
	Since
	\begin{align*}
	\int_{[x,\infty)} \bigg( \int_0^t g \bigg)^{\alpha + 1}\,d[-f(t)] & \ge \bigg( \int_0^x g \bigg)^{\alpha + 1} \int_{[x,\infty)} \,d[-f(t)] \\
	& = \bigg( \int_0^x g \bigg)^{\alpha + 1} [f(x) - \lim_{x \rw +\infty} f(x)] = f(x) \,\bigg( \int_0^x g \bigg)^{\alpha + 1}, \quad x>0,
	\end{align*}
	we obtain that
	$$
	f(x) \,\bigg( \int_0^x g \bigg)^{\alpha + 1} \rightarrow 0, \quad \mbox{as} \quad x \rightarrow +\infty.
	$$
	Thus, integrating by parts, we get that	
	\begin{align*}
	A_1 & = \int_0^{\infty} \bigg( \int_0^t g \bigg)^{\alpha} g(t) f(t)\,dt \approx \int_0^{\infty} f(t) \,d \bigg(\int_0^t g \bigg)^{\alpha + 1} \\
	& = f(t) \, \bigg(\int_0^t g \bigg)^{\alpha + 1} \bigg|_0^{\infty} + \int_0^{\infty} \bigg(\int_0^t g \bigg)^{\alpha + 1} \,d [-f(t)] \\
	& = \lim_{t \rw \infty} f(t) \, \bigg(\int_0^t g \bigg)^{\alpha + 1} - \lim_{t \rw 0+} f(t) \, \bigg(\int_0^t g \bigg)^{\alpha + 1}
	+ \int_0^{\infty} \bigg(\int_0^t g \bigg)^{\alpha + 1} \,d [-f(t)] \\
	& \le \int_0^{\infty} \bigg(\int_0^t g \bigg)^{\alpha + 1} \,d [-f(t)] = A_2.
	\end{align*}
	Hence
	$$
	A_1 \lesssim A_2.
	$$
	
	We have shown that if $\lim_{x \rw +\infty} f(x) = 0$, then
	$$
	A_1 < \infty \quad \Longleftrightarrow \quad A_2 < \infty,
	$$
	and 
	$$
	A_1 \approx A_2.
	$$
	
	Now assume that $\lim_{x \rw +\infty} f(x) > 0$. Then, applying previous statement to the function $f(x) - \lim_{x \rw +\infty} f(x)$, we arrive at
	$$
	\int_0^{\infty} \bigg( \int_0^t g \bigg)^{\alpha} g(t) [f(t) - \lim_{x \rw +\infty} f(x)]\,dt \approx\int_{(0,\infty)}  
	\bigg( \int_0^t g \bigg)^{\alpha + 1}\,d[-f(t)].
	$$
	
	The proof is completed.
\end{proof}

\begin{rem}\label{rem.IBP.0}
	Note that if $f \in \mp^{+,\dn}(0,\infty)$ is such that $\lim_{x \rw +\infty} f(x) > 0$, then
	$$
	\int_0^{\infty} \bigg( \int_0^t g \bigg)^{\alpha} g(t) f(t)\,dt  < \infty \quad \Longrightarrow \quad \int_0^{\infty} g(x)\,dx < \infty.
	$$
	Indeed: for each $x \in (0,\infty)$
	\begin{align*}
	\infty > \int_0^{\infty} \bigg( \int_0^t g \bigg)^{\alpha} g(t) f(t)\,dt & \ge \int_0^x \bigg( \int_0^t g \bigg)^{\alpha} g(t) f(t)\,dt \\
	& \ge f(x) \, \int_0^x \bigg( \int_0^t g \bigg)^{\alpha} g(t)\,dt \approx f(x)\,\bigg(\int_0^x g \bigg)^{\alpha + 1} 	
	\end{align*}
	holds. Thus
	$$
	\lim_{x \rw +\infty} f(x) \cdot \bigg(\int_0^x g \bigg)^{\alpha + 1} \le f(x)\,\bigg(\int_0^x g \bigg)^{\alpha + 1} \le \int_0^{\infty} \bigg( \int_0^t g \bigg)^{\alpha} g(t) f(t)\,dt < \infty.
	$$
	Hence
	$$
	\lim_{x \rw +\infty} f(x) \cdot \bigg(\int_0^{\infty} g \bigg)^{\alpha + 1} < \infty.
	$$
	Therefore
	$$
	\int_0^{\infty} g < \infty.
	$$
\end{rem}

\begin{cor}\label{cor.IBP.0}
	Let $\alpha > 0$. Let $g$ be a non-negative function on $(0,\infty)$ such that $0 < \int_0^t g < \infty$, $t > 0$ and let $f$ be a non-negative non-increasing right-continuous function on $(0,\infty)$. Then
	\begin{align*}
	\int_0^{\infty} \bigg( \int_0^t g \bigg)^{\alpha} g(t) f(t)\,dt \approx \int_{(0,\infty)}  
	\bigg( \int_0^t g \bigg)^{\alpha + 1}\,d[-f(t)] + \lim_{x \rw +\infty} f(x) \cdot \bigg( \int_0^{\infty} g \bigg)^{\alpha + 1}.
	\end{align*}
\end{cor}

\begin{proof}
	If $\lim_{x \rw +\infty} f(x) = 0$, then the statement follows by Theorem \ref{thm.IBP.0}. If $\lim_{x \rw +\infty} f(x) > 0$, then by Remark \ref{rem.IBP.0}, we know that	
	$$
	\int_0^{\infty} \bigg( \int_0^t g \bigg)^{\alpha} g(t) f(t)\,dt < \infty \quad \Longrightarrow \quad \int_0^{\infty} g(x)\,dx < \infty.
	$$
	Therefore, by Theorem \ref{thm.IBP.0}, we get that
	\begin{align*}
	\int_0^{\infty} \bigg( \int_0^t g \bigg)^{\alpha} g(t) f(t)\,dt & = \int_0^{\infty} \bigg( \int_0^t g \bigg)^{\alpha} g(t) [f(t) - \lim_{x \rw +\infty} f(x)]\,dt + \lim_{x \rw +\infty} f(x) \cdot \int_0^{\infty} \bigg( \int_0^t g \bigg)^{\alpha} g(t)\,dt \\
	& \approx \int_{(0,\infty)} \bigg( \int_0^t g \bigg)^{\alpha + 1}\,d[-f(t)] + \lim_{x \rw +\infty} f(x) \cdot \bigg( \int_0^{\infty} g \bigg)^{\alpha + 1}.
	\end{align*}
	
	The proof is completed.
\end{proof}

\begin{thm}\label{thm.IBP}
	Let $\alpha > 0$. Let $g$ be a non-negative function on $(0,\infty)$ such that $0 < \int_t^{\infty} g < \infty$, $t > 0$ and let $f$ be a non-negative non-decreasing left-continuous function on $(0,\infty)$. Then
	\begin{align*}
	B_1 : = \int_0^{\infty} \bigg( \int_t^{\infty} g \bigg)^{\alpha} g(t) [f(t) - f(0+)]\,dt < \infty \quad \Longleftrightarrow \quad B_2 : = \int_{(0,\infty)}  
	\bigg( \int_t^{\infty} g \bigg)^{\alpha + 1}\,d[f(t)] < \infty.
	\end{align*}
	Moreover, $B_1 \approx B_2$.
\end{thm}

\begin{proof}
	Assume at first that $f(0+) = 0$. Let 
	$$
	B_1 : = \int_0^{\infty} \bigg( \int_t^{\infty} g \bigg)^{\alpha} g(t) f(t)\,dt < \infty.
	$$
	Then
	$$
	\int_x^{\infty} \bigg( \int_t^{\infty} g \bigg)^{\alpha} g(t) f(t)\,dt \rightarrow 0, \quad \mbox{as} \quad x \rightarrow \infty.
	$$
	Since
	\begin{align*}
	\int_x^{\infty} \bigg( \int_t^{\infty} g \bigg)^{\alpha} g(t) f(t)\,dt \ge f(x) \, \int_x^{\infty} \bigg( \int_t^{\infty} g \bigg)^{\alpha} g(t) \,dt \approx f(x) \, \bigg( \int_x^{\infty} g \bigg)^{\alpha + 1},  \quad x > 0, 
	\end{align*}
	we have that
	$$
	f(x) \, \bigg( \int_x^{\infty} g \bigg)^{\alpha + 1} \rightarrow 0, \quad \mbox{as} \quad x \rightarrow \infty.
	$$
	Hence, integrating by parts, we get that
	\begin{align*}
	B_2 & = \int_{(0,\infty)} \bigg( \int_t^{\infty} g \bigg)^{\alpha + 1}\,d[f(t)] = f(t) \, \bigg( \int_t^{\infty} g \bigg)^{\alpha + 1} \bigg|_{0}^{\infty} - 
	\int_{(0,\infty)} f(t)\,d\bigg( \int_t^{\infty} g \bigg)^{\alpha + 1} \\
	& = \lim_{t \rw \infty} f(t) \, \bigg( \int_t^{\infty} g \bigg)^{\alpha + 1} - \lim_{t \rw 0+} f(t) \, \bigg( \int_t^{\infty} g \bigg)^{\alpha + 1} + (\alpha + 1) \int_0^{\infty} \bigg( \int_t^{\infty} g \bigg)^{\alpha} g(t) f(t)\,dt \\
	& \le (\alpha + 1) \int_0^{\infty} \bigg( \int_t^{\infty} g \bigg)^{\alpha} g(t) f(t)\,dt = (\alpha + 1) B_1.
	\end{align*}
	
	Now assume that 	
	$$
	B_2 : = \int_{(0,\infty)} \bigg( \int_t^{\infty} g \bigg)^{\alpha + 1}\,d[f(t)] < \infty.
	$$
	Then
	$$
	\int_{(0,x]} \bigg( \int_t^{\infty} g \bigg)^{\alpha + 1}\,d[f(t)] \rightarrow 0, \quad \mbox{as} \quad x \rightarrow 0+.
	$$
	Since
	\begin{align*}
	\int_{(0,x]} \bigg( \int_t^{\infty} g \bigg)^{\alpha + 1}\,d[f(t)] & \ge \bigg( \int_x^{\infty} g \bigg)^{\alpha + 1} \int_{(0,x]} \,d[f(t)] \\
	& = \bigg( \int_x^{\infty} g \bigg)^{\alpha + 1} [f(x) - f(0+)] = f(x) \,\bigg( \int_x^{\infty} g \bigg)^{\alpha + 1}, \quad x>0,
	\end{align*}
	we obtain that
	$$
	f(x) \,\bigg( \int_x^{\infty} g \bigg)^{\alpha + 1} \rightarrow 0, \quad \mbox{as} \quad x \rightarrow 0+.
	$$
	Thus, integrating by parts, we get that	
	\begin{align*}
	B_1 & = \int_0^{\infty} \bigg( \int_t^{\infty} g \bigg)^{\alpha} g(t) f(t)\,dt \approx \int_0^{\infty} f(t)\,d \bigg[ - \bigg(\int_t^{\infty} g \bigg)^{\alpha + 1} \bigg] \\
	& = - f(t) \, \bigg(\int_t^{\infty} g \bigg)^{\alpha + 1} \bigg|_0^{\infty} + \int_0^{\infty} \bigg(\int_t^{\infty} g \bigg)^{\alpha + 1} \,d [f(t)] \\
	& = \lim_{t \rw 0+} f(t) \, \bigg(\int_t^{\infty} g \bigg)^{\alpha + 1} - \lim_{t \rw \infty} f(t) \, \bigg(\int_t^{\infty} g \bigg)^{\alpha + 1}
	+ \int_0^{\infty} \bigg(\int_t^{\infty} g \bigg)^{\alpha + 1} \,d [f(t)] \\
	& \le \int_0^{\infty} \bigg(\int_t^{\infty} g \bigg)^{\alpha + 1} \,d [f(t)] = B_2.
	\end{align*}
	
	We have shown that if $f(0+) = 0$, then
	$$
	B_1 < \infty \quad \Longleftrightarrow \quad B_2 < \infty,
	$$
	and 
	$$
	B_1 \approx B_2.
	$$
	
	Now assume that $f(0+) > 0$. Then, applying previous statement to the function $f(x) - f(0+)$, we arrive at
	$$
	\int_0^{\infty} \bigg( \int_t^{\infty} g \bigg)^{\alpha} g(t) [f(t) - f(0+)]\,dt \approx\int_{(0,\infty)}  
	\bigg( \int_t^{\infty} g \bigg)^{\alpha + 1}\,d[f(t)].
	$$
	
	The proof is completed.
\end{proof}

\begin{rem}\label{rem.IBP}
	Note that if $f$ is a non-negative non-decreasing function on $(0,\infty)$ such that $f(0+) > 0$, then
	$$
	\int_0^{\infty} \bigg( \int_t^{\infty} g \bigg)^{\alpha} g(t) f(t)\,dt  < \infty \quad \Longrightarrow \quad \int_0^{\infty} g(x)\,dx < \infty.
	$$
	Indeed: for each $x \in (0,\infty)$
	\begin{align*}
	\infty > \int_0^{\infty} \bigg( \int_t^{\infty} g \bigg)^{\alpha} g(t) f(t)\,dt & \ge \int_x^{\infty} \bigg( \int_t^{\infty} g \bigg)^{\alpha} g(t) f(t)\,dt \\
    & \ge f(x) \, \int_x^{\infty} \bigg( \int_t^{\infty} g \bigg)^{\alpha} g(t)\,dt \approx f(x)\,\bigg(\int_x^{\infty} g \bigg)^{\alpha + 1} 	
	\end{align*}
	holds. Thus
	$$
	f(0+) \bigg(\int_x^{\infty} g \bigg)^{\alpha + 1} \le f(x)\,\bigg(\int_x^{\infty} g \bigg)^{\alpha + 1} \le \int_0^{\infty} \bigg( \int_t^{\infty} g \bigg)^{\alpha} g(t) f(t)\,dt < \infty.
	$$
	Hence
	$$
	f(0+) \, \bigg(\int_0^{\infty} g \bigg)^{\alpha + 1} < \infty.
	$$
	Therefore
	$$
	\int_0^{\infty} g < \infty.
	$$
\end{rem}

\begin{cor}\label{cor.IBP}
	Let $\alpha > 0$. Let $g$ be a non-negative function on $(0,\infty)$ such that $0 < \int_t^{\infty} g < \infty$, $t > 0$ and let $f$ be a non-negative non-decreasing left-continuous function on $(0,\infty)$. Then
	\begin{align*}
	\int_0^{\infty} \bigg( \int_t^{\infty} g \bigg)^{\alpha} g(t) f(t)\,dt \approx \int_{(0,\infty)}  
	\bigg( \int_t^{\infty} g \bigg)^{\alpha + 1}\,d[f(t)] + f(0+) \, \bigg( \int_0^{\infty} g \bigg)^{\alpha + 1}.
	\end{align*}
\end{cor}

\begin{proof}
If $f(0+) = 0$, then the statement follows by Theorem \ref{thm.IBP}. If $f(0+) > 0$, then by Remark \ref{rem.IBP}, we know that	
$$
\int_0^{\infty} \bigg( \int_t^{\infty} g \bigg)^{\alpha} g(t) f(t)\,dt < \infty \quad \Longrightarrow \quad \int_0^{\infty} g(x)\,dx < \infty.
$$
Therefore, by Theorem \ref{thm.IBP}, we get that
\begin{align*}
\int_0^{\infty} \bigg( \int_t^{\infty} g \bigg)^{\alpha} g(t) f(t)\,dt & = \int_0^{\infty} \bigg( \int_t^{\infty} g \bigg)^{\alpha} g(t) [f(t) - f(0+)]\,dt + f(0+) \,\int_0^{\infty} \bigg( \int_t^{\infty} g \bigg)^{\alpha} g(t)\,dt \\
& \approx \int_{(0,\infty)} \bigg( \int_t^{\infty} g \bigg)^{\alpha + 1}\,d[f(t)] + f(0+) \, \bigg( \int_0^{\infty} g \bigg)^{\alpha + 1}.
\end{align*}

The proof is completed.
\end{proof}


\

\section{The boundedness of $T_{u,b}$ and $T_{u,b}^*$ from $L^p(v)$ into $\ces_q(w,a)$}\label{main results}

\


	In this section we give solutions of the following two inequalities 
	\begin{align}
	\bigg( \int_0^{\infty} \bigg( \int_0^x \bigg( \sup_{t \le \tau}
	\frac{u(\tau)}{B(\tau)} \int_0^{\tau} h(y)b(y)\,dy \bigg) \,a(t)\,dt \bigg)^q w(x)\,dx \bigg)^{\frac{1}{q}} \le C \, \bigg( \int_0^{\infty} h(s)^pv(s)\,ds\bigg)^{\frac{1}{p}}, \quad h \in \mp^+ (0,\infty) \label{eq.1}
	\end{align}
	and
	\begin{align}
	\bigg( \int_0^{\infty} \bigg( \int_0^x \bigg( \sup_{t \le \tau}
	\frac{u(\tau)}{B(\tau)} \int_{\tau}^{\infty} h(y)b(y)\,dy \bigg) \,a(t)\,dt \bigg)^q w(x)\,dx \bigg)^{\frac{1}{q}} \le C \, \bigg( \int_0^{\infty} h(s)^pv(s)\,ds\bigg)^{\frac{1}{p}}, \quad h \in \mp^+ (0,\infty), \label{eq.2}
	\end{align}
	where $1 < p \le q < \infty$ and $a,\,u,\,v,\,w \in \W \I$. Using the duality argument, we reduce the problem to the boundedness for the dual of integral Volterra operator with a kernel satisfying Oinarov’s condition and weighted Stieltjes operator.
	
	Note that the characterization of inequalities
	\begin{align}
	\bigg( \int_0^{\infty} \bigg( \int_x^{\infty} \bigg( \sup_{t \le \tau}
	\frac{u(\tau)}{B(\tau)} \int_0^{\tau} h(y)b(y)\,dy \bigg) \,a(t)\,dt \bigg)^q w(x)\,dx \bigg)^{\frac{1}{q}} \le C \, \bigg( \int_0^{\infty} h(s)^pv(s)\,ds\bigg)^{\frac{1}{p}}, \quad h \in \mp^+ (0,\infty) \label{eq.3}
	\end{align}
	and
	\begin{align}
	\bigg( \int_0^{\infty} \bigg( \int_x^{\infty} \bigg( \sup_{t \le \tau}
	\frac{u(\tau)}{B(\tau)} \int_{\tau}^{\infty} h(y)b(y)\,dy \bigg) \,a(t)\,dt \bigg)^q w(x)\,dx \bigg)^{\frac{1}{q}} \le C \, \bigg( \int_0^{\infty} h(s)^pv(s)\,ds\bigg)^{\frac{1}{p}}, \quad h \in \mp^+ (0,\infty) \label{eq.4}
	\end{align}
	can be reduced to the solutions of \eqref{eq.1} and \eqref{eq.2}.
	
	Recall that, if $F$ is a non-negative non-decreasing function on $\I$, then
	\begin{equation}\label{Fubini.2}	
	\esup_{t \in (0,\infty)} F(t)G(t) = \esup_{t \in (0,\infty)} F(t) \esup_{\tau \in (t,\infty)} G(\tau),
	\end{equation}
	likewise, when $F$ is a non-negative non-increasing function on $\I$, then
	\begin{equation}\label{Fubini.1}
	\esup_{t \in (0,\infty)} F(t)G(t) = \esup_{t \in (0,\infty)} F(t)
	\esup_{\tau \in (0,t)} G(\tau)
	\end{equation}
	(see, for instance, \cite[p. 85]{gp2}).
	
	We need the following notations:
	$$
	\begin{array}{ll}
	A(t) : =\int_0^t a(s)ds, \qq U(t) : =\int_0^t u(s)ds, \qq W(t) : =\int_0^t w(s)ds.
	\end{array}
	$$
			
	\begin{thm}\label{aux.thm.1}
		Let $1 < p,\, q < \infty$. Assume that $u \in \W\I \cap C\I$ and $a,\,v,\,w \in \W\I$. Moreover, assume that
		$$
		0 < \int_0^x v(t)^{1-p'}\,dt < \infty \qquad \mbox{for all} \quad x > 0.
		$$ 
		
		{\rm (i)} If $p \le q$, then
		\begin{align*}
		\sup_{h \ge  0} \frac{\bigg( \int_0^{\infty} \bigg( \int_0^x \bigg( \sup_{t \le \tau}u(\tau) \int_0^{\tau} h(y)\,dy \bigg) a(t)\,dt \bigg)^q w(x)\,dx \bigg)^{\frac{1}{q}}}{\bigg( \int_0^{\infty} h(s)^pv(s)\,ds\bigg)^{\frac{1}{p}}} & \\
		& \hspace{-7cm} \approx \, \sup_{t \in (0,\infty)} \bigg( \int_0^t v(x)^{1-p'} \, \bigg( \int_x^t \bigg( \sup_{s \le \tau}u(\tau)\bigg) a(s)\,ds \bigg)^{p'} \, dx \bigg)^{\frac{1}{p'}} \, \bigg( \int_t^{\infty} w(y) \,dy \bigg)^{\frac{1}{q}} \notag \\
		& \hspace{-6.5cm} + \sup_{t \in (0,\infty)} \bigg( \int_0^t v(x)^{1-p'}  \, dx \bigg)^{\frac{1}{p'}} \, \bigg( \int_t^{\infty} \, \bigg( \int_t^y \bigg( \sup_{s \le \tau}u(\tau)\bigg) a(s)\,ds \bigg)^{q} w(y) \,dy \bigg)^{\frac{1}{q}} \\
		& \hspace{-6.5cm} + \, \sup_{x \in (0,\infty)} \bigg( \int_{[x,\infty)} \, d \, \bigg( - \sup_{t \le \tau} u(\tau)^{p'} \bigg( \int_0^{\tau} v(s)^{1-p'}\,ds \bigg) \bigg) \bigg)^{\frac{1}{p'}} \, \bigg( \int_0^x A(y)^q w(y) \,dy \bigg)^{\frac{1}{q}} \notag \\
		& \hspace{-6.5cm} + \sup_{x \in (0,\infty)} \bigg( \int_{(0,x]} \,  A(t)^{p'} \, d \, \bigg( - \sup_{t \le \tau} u(\tau)^{p'} \bigg( \int_0^{\tau} v(s)^{1-p'}\,ds \bigg) \bigg) \bigg)^{\frac{1}{p'}} \, \bigg( \int_x^{\infty} w(y) \,dy \bigg)^{\frac{1}{q}} \notag \\
		& \hspace{-6.5cm} + \bigg( \int_0^{\infty} A(y)^q w(y) \,dy \bigg)^{\frac{1}{q}} \, \lim_{t \rightarrow \infty} \bigg(\sup_{t \le \tau} u(\tau) \bigg( \int_0^{\tau} v(s)^{1-p'}\,ds \bigg)^{\frac{1}{p'}}\bigg);
		\end{align*}
				
		{\rm (ii)} If $q < p$, then
		\begin{align*}
		\sup_{h \ge  0} \frac{\bigg( \int_0^{\infty} \bigg( \int_0^x \bigg( \sup_{t \le \tau}u(\tau) \int_0^{\tau} h(y)\,dy \bigg) a(t)\,dt \bigg)^q w(x)\,dx \bigg)^{\frac{1}{q}}}{\bigg( \int_0^{\infty} h(s)^pv(s)\,ds\bigg)^{\frac{1}{p}}} & \\
		& \hspace{-7cm} \approx \, \bigg( \int_0^{\infty}  \bigg( \int_0^t v(x)^{1-p'} \,dx \bigg)^{\frac{r}{q'}} v(t)^{1-p'} \, \bigg( \int_t^{\infty} \bigg( \int_t^z \bigg( \sup_{s \le y}u(y)\bigg) a(s)\,ds \bigg)^{q} w(z)\, dz \bigg)^{\frac{r}{q}} \, dt \bigg)^{\frac{1}{r}} \notag \\
		& \hspace{-6.5cm} + \bigg( \int_0^{\infty} \bigg( \int_0^t v(x)^{1-p'} \bigg( \int_x^t \bigg( \sup_{s \le y}u(y)\bigg) a(s)\,ds \bigg)^{p'} \, dx \bigg)^{\frac{r}{p'}} \bigg( \int_z^{\infty} w(s) \,ds \bigg)^{\frac{r}{p}} w(t)\,dt \bigg)^{\frac{1}{r}} \\
		& \hspace{-6.5cm} + \, \bigg( \int_0^{\infty} \bigg( \int_{[x,\infty)} \, d \, \bigg( - \bigg(\sup_{t \le \tau} u(\tau)^{p'} \, \bigg( \int_0^{\tau} v(s)^{1-p'}\,ds \bigg) \bigg) \bigg) \bigg)^{\frac{r}{p'}} 
		\, \bigg( \int_0^x A(y)^q w(y) \,dy \bigg)^{\frac{r}{p}} A(x)^q w(x) \,dx\bigg)^{\frac{1}{r}} \notag \\
		& \hspace{-6.5cm} + \bigg( \int_0^{\infty} \bigg( \int_{(0,x]} \,  A(t)^{p'} \, d \, \bigg( - \bigg(\sup_{t \le \tau} u(\tau)^{p'} \, \bigg( \int_0^{\tau} v(s)^{1-p'}\,ds \bigg) \bigg) \bigg)  \bigg)^{\frac{r}{p'}} \, \bigg( \int_x^{\infty} w(y) \,dy \bigg)^{\frac{r}{p}} w(x)\,dx \bigg)^{\frac{1}{r}} \notag \\
		& \hspace{-6.5cm} + \bigg( \int_0^{\infty} A(y)^q w(y) \,dy \bigg)^{\frac{1}{q}} \, \lim_{t \rightarrow \infty} \bigg(\sup_{t \le \tau} u(\tau) \bigg( \int_0^{\tau} v(s)^{1-p'}\,ds \bigg)^{\frac{1}{p'}}\bigg). 		
		\end{align*}
	\end{thm}
	
	\begin{proof}
		Assume that $1 < p \le q < \infty$. 
		By duality, using Fubini's Theorem, and interchanging the suprema, we get that 
		\begin{align*}
		\sup_{h \ge  0} \frac{\bigg( \int_0^{\infty} \bigg( \int_0^x \bigg( \sup_{t \le \tau}
			u(\tau) \int_0^{\tau} h(y)\,dy \bigg) a(t)\,dt \bigg)^q w(x)\,dx \bigg)^{\frac{1}{q}}}{\bigg( \int_0^{\infty} h(s)^pv(s)\,ds\bigg)^{\frac{1}{p}}} & \\
		& \hspace{-5cm} = \sup_{h \ge  0} \frac{1}{\bigg( \int_0^{\infty} h(s)^pv(s)\,ds\bigg)^{\frac{1}{p}}} \sup_{g \ge  0}\frac{ \int_0^{\infty} \bigg( \int_0^x \bigg( \sup_{t \le \tau}	u(\tau) \int_0^{\tau} h(y)\,dy \bigg) a(t)\, dt \bigg) g(x)\,dx }{\bigg( \int_0^{\infty} g(x)^{q'}w(x)^{1-q'}\,dx\bigg)^{\frac{1}{q'}}} \\
		& \hspace{-5cm} = \sup_{g \ge  0} \frac{1}{\bigg( \int_0^{\infty} g(x)^{q'}w(x)^{1-q'}\,dx\bigg)^{\frac{1}{q'}}} \sup_{h \ge  0}\frac{ \int_0^{\infty}  \bigg( \sup_{t \le \tau}	u(\tau) \int_0^{\tau} h(y)\,dy \bigg) \bigg( \int_t^{\infty} g(x)\,dx \bigg) a(t)\,dt }{\bigg( \int_0^{\infty} h(s)^pv(s)\,ds\bigg)^{\frac{1}{p}}}.
		\end{align*}
		
		Applying \cite[Theorems 4.4]{gop}, on using \eqref{Fubini.2}, we arrive at
		\begin{align*}
		\sup_{h \ge  0}\frac{ \int_0^{\infty}  \bigg( \sup_{t \le \tau}	u(\tau) \int_0^{\tau} h(y)\,dy \bigg) \bigg( \int_t^{\infty} g(x)\,dx \bigg) a(t)\,dt }{\bigg( \int_0^{\infty} h(s)^pv(s)\,ds\bigg)^{\frac{1}{p}}} & \approx D + E,
		\end{align*}
		where
		\begin{align*}
		D & : = \bigg( \int_0^{\infty} \bigg( \int_t^{\infty} \bigg( \sup_{s \le \tau}u(\tau)\bigg) \bigg( \int_s^{\infty} g(x)\,dx \bigg) a(s)\,ds \bigg)^{\frac{p'}{p}} \bigg( \sup_{t \le \tau}u(\tau)\bigg) \bigg( \int_0^t v(s)^{1-p'}\,ds \bigg) \bigg( \int_t^{\infty} g(x)\,dx \bigg)a(t)\,dt \bigg)^{\frac{1}{p'}}, \\
		E & : = \bigg( \int_0^{\infty} \bigg( \int_0^t \bigg( \int_s^{\infty} g(x)\,dx \bigg) a(s)\,ds \bigg)^{\frac{p'}{p}} \bigg(\sup_{t \le \tau} u(\tau)^{p'} \bigg( \int_0^{\tau} v(s)^{1-p'}\,ds \bigg) \bigg) \bigg( \int_t^{\infty} g(x)\,dx \bigg)a(t)\,dt \bigg)^{\frac{1}{p'}}.
		\end{align*}
		
		Integrating by parts (applying Corollary \ref{cor.IBP}), on using Fubini's Theorem, we arrive at
		\begin{align*}
		D & \approx \bigg( \int_0^{\infty} \bigg( \int_t^{\infty} \bigg( \sup_{s \le \tau}u(\tau)\bigg) \bigg( \int_s^{\infty} g(x)\,dx \bigg) a(s)\,ds \bigg)^{p'} \, v(t)^{1-p'}\,dt  \bigg)^{\frac{1}{p'}} \\
        & = \bigg( \int_0^{\infty} \bigg( \int_t^{\infty} g(x) \int_t^x \bigg( \sup_{s \le \tau}u(\tau)\bigg) a(s)\,ds \,dx \bigg)^{p'} \, v(t)^{1-p'} \,dt \bigg)^{\frac{1}{p'}}.
		\end{align*}
		Similarly, integrating by parts (applying Corollary \ref{cor.IBP.0}), on using Fubini's Theorem, we get at
		\begin{align*}
		E \approx & \,\, \bigg( \int_0^{\infty} \bigg(\sup_{t \le \tau} u(\tau)^{p'} \,  \bigg(\int_0^{\tau} v(s)^{1-p'}\,ds \bigg) \bigg)  \,d \, \bigg( \int_0^t \bigg( \int_s^{\infty} g(x)\,dx \bigg) a(s)\,ds \bigg)^{p'} \bigg)^{\frac{1}{p'}} \\
		\approx &  \,\, \bigg( \int_0^{\infty} \bigg( \int_0^t \bigg( \int_s^{\infty} g(x)\,dx \bigg) a(s)\,ds \bigg)^{p'}
		\,d \, \bigg( - \bigg(\sup_{t \le \tau} u(\tau)^{p'} \, \bigg( \int_0^{\tau} v(s)^{1-p'}\,ds \bigg) \bigg) \bigg)  \bigg)^{\frac{1}{p'}} \\
		& + \bigg( \int_0^{\infty} \bigg( \int_s^{\infty} g(x)\,dx \bigg) a(s)\,ds \bigg) \lim_{t \rightarrow \infty}
		\bigg(\sup_{t \le \tau} u(\tau)^{p'} \, \bigg( \int_0^{\tau} v(s)^{1-p'}\,ds \bigg) \bigg) \\
		\approx &  \,\, \bigg( \int_0^{\infty} \bigg( \int_0^t  g(x)A(x)\,dx \bigg)^{p'}
		\,d \, \bigg( - \bigg(\sup_{t \le \tau} u(\tau)^{p'} \, \bigg( \int_0^{\tau} v(s)^{1-p'}\,ds \bigg) \bigg) \bigg) \bigg)^{\frac{1}{p'}} \\
		& + \bigg( \int_0^{\infty} \bigg( \int_t^{\infty} g(x)\,dx \bigg)^{p'} A(t)^{p'}
		\,d \, \bigg( - \bigg(\sup_{t \le \tau} u(\tau)^{p'} \, \bigg( \int_0^{\tau} v(s)^{1-p'}\,ds \bigg) \bigg) \bigg)  \bigg)^{\frac{1}{p'}} \\
		& + \bigg( \int_0^{\infty} g(x)A(x)\,dx  \bigg) \lim_{t \rightarrow \infty}
		\bigg(\sup_{t \le \tau} u(\tau) \, \bigg( \int_0^{\tau} v(s)^{1-p'}\,ds \bigg)^{\frac{1}{p'}} \bigg) : = E_1 + E_2 + E_3.
		\end{align*}
		
		{\rm (i)} Let $p \le q$. By \cite[Theorem 1.1]{Oinar}, we obtain that
		\begin{align}
		\sup_{g \ge  0} \frac{ D }{\bigg( \int_0^{\infty} g^{q'}w^{1-q'}\bigg)^{\frac{1}{q'}}} & \notag \\
		& \hspace{-3cm} = \, \sup_{g \ge  0} \frac{1}{\bigg( \int_0^{\infty} g^{q'}w^{1-q'}\bigg)^{\frac{1}{q'}}}
		\bigg( \int_0^{\infty} \bigg( \int_t^{\infty} g(x) \int_t^x \bigg( \sup_{s \le y}u(y)\bigg) a(s)\,ds \,dx \bigg)^{p'} \, v(t)^{1-p'} \,dt \bigg)^{\frac{1}{p'}} \notag \\
		& \hspace{-3cm} \approx \, \sup_{t \in (0,\infty)} \bigg( \int_0^t v(x)^{1-p'} \, \bigg( \int_x^t \bigg( \sup_{s \le y}u(y)\bigg) a(s)\,ds \bigg)^{p'} \, dx \bigg)^{\frac{1}{p'}} \, \bigg( \int_t^{\infty} w(y) \,dy \bigg)^{\frac{1}{q}} \notag \\
		& \hspace{-2.5cm} + \sup_{t \in (0,\infty)} \bigg( \int_0^t v(x)^{1-p'}  \, dx \bigg)^{\frac{1}{p'}} \, \bigg( \int_t^{\infty} \, \bigg( \int_t^z \bigg( \sup_{s \le y}u(y)\bigg) a(s)\,ds \bigg)^{q} w(z) \,dz \bigg)^{\frac{1}{q}}. \label{eq.I}
		\end{align}
		
		By \cite[Theorem 1, p. 40 and Theorem 3, p. 44]{mazya}, respectively, we have that 
		\begin{align*}
		\sup_{g \ge  0} \frac{ E_1 }{\bigg( \int_0^{\infty} g^{q'}w^{1-q'}\bigg)^{\frac{1}{q'}}} & \\
		& \hspace{-3cm} = \sup_{g \ge  0} \frac{\bigg( \int_0^{\infty} \bigg( \int_0^t  g(x)A(x)\,dx \bigg)^{p'} \,d \, \bigg( - \bigg(\sup_{t \le \tau} u(\tau)^{p'} \, \bigg( \int_0^{\tau} v(s)^{1-p'}\,ds \bigg) \bigg) \bigg)  \bigg)^{\frac{1}{p'}}}{\bigg( \int_0^{\infty} g^{q'}w^{1-q'}\bigg)^{\frac{1}{q'}}}
		 \\
		& \hspace{-3cm} \approx \sup_{x \in (0,\infty)} \bigg( \int_{[x,\infty)} \, d \, \bigg( - \bigg(\sup_{t \le \tau} u(\tau)^{p'} \, \bigg( \int_0^{\tau} v(s)^{1-p'}\,ds \bigg) \bigg) \bigg) \bigg)^{\frac{1}{p'}} \, \bigg( \int_0^x A(y)^q w(y) \,dy \bigg)^{\frac{1}{q}}
		\end{align*}
		and
		\begin{align*}
		\sup_{g \ge  0} \frac{ E_2 }{\bigg( \int_0^{\infty} g^{q'}w^{1-q'}\bigg)^{\frac{1}{q'}}} & \\
		& \hspace{-3cm} = \sup_{g \ge  0} \frac{\bigg( \int_0^{\infty} \bigg( \int_t^{\infty} g(x)\,dx \bigg)^{p'} A(t)^{p'}
		\,d \, \bigg( - \bigg(\sup_{t \le \tau} u(\tau)^{p'} \, \bigg( \int_0^{\tau} v(s)^{1-p'}\,ds \bigg) \bigg) \bigg)  \bigg)^{\frac{1}{p'}}}{\bigg( \int_0^{\infty} g^{q'}w^{1-q'}\bigg)^{\frac{1}{q'}}} \\
		& \hspace{-3cm} \approx \sup_{x \in (0,\infty)} \bigg( \int_{(0,x]} \,  A(t)^{p'} \, d \, \bigg( - \bigg(\sup_{t \le \tau} u(\tau)^{p'} \, \bigg( \int_0^{\tau} v(s)^{1-p'}\,ds \bigg) \bigg) \bigg)  \bigg)^{\frac{1}{p'}} \, \bigg( \int_x^{\infty} w(y) \,dy \bigg)^{\frac{1}{q}}.
		\end{align*}
		
		By duality, we have that
		\begin{align*}
		\sup_{g \ge  0} \frac{ E_3 }{\bigg( \int_0^{\infty} g^{q'}w^{1-q'}\bigg)^{\frac{1}{q'}}} & \\
		& \hspace{-3cm} = \sup_{g \ge  0} \frac{\int_0^{\infty} g(x)A(x)\,dx }{\bigg( \int_0^{\infty} g^{q'}w^{1-q'}\bigg)^{\frac{1}{q'}}} \cdot \lim_{t \rightarrow \infty} \bigg(\sup_{t \le \tau} u(\tau) \bigg( \int_0^{\tau} v(s)^{1-p'}\,ds \bigg)^{\frac{1}{p'}}\bigg) \\
		& \hspace{-3cm} = \bigg( \int_0^{\infty} A(y)^q w(y) \,dy \bigg)^{\frac{1}{q}} \, \lim_{t \rightarrow \infty} \bigg(\sup_{t \le \tau} u(\tau) \bigg( \int_0^{\tau} v(s)^{1-p'}\,ds \bigg)^{\frac{1}{p'}}\bigg).
		\end{align*}
		
		Thus, we get that
		\begin{align}
		\sup_{g \ge  0} \frac{E}{\bigg( \int_0^{\infty} g^{q'}w^{1-q'}\bigg)^{\frac{1}{q'}}} & \notag \\
		& \hspace{-3cm} \approx \, \sup_{x \in (0,\infty)} \bigg( \int_{[x,\infty)} \, d \, \bigg( - \sup_{t \le \tau} u(\tau)^{p'} \bigg( \int_0^{\tau} v(s)^{1-p'}\,ds \bigg) \bigg) \bigg)^{\frac{1}{p'}} \, \bigg( \int_0^x A(y)^q w(y) \,dy \bigg)^{\frac{1}{q}} \notag \\
		& \hspace{-2.5cm} + \sup_{x \in (0,\infty)} \bigg( \int_{(0,x]} \,  A(t)^{p'} \, d \, \bigg( - \sup_{t \le \tau} u(\tau)^{p'} \bigg( \int_0^{\tau} v(s)^{1-p'}\,ds \bigg) \bigg) \bigg)^{\frac{1}{p'}} \, \bigg( \int_x^{\infty} w(y) \,dy \bigg)^{\frac{1}{q}} \notag \\
		& \hspace{-2.5cm} + \bigg( \int_0^{\infty} A(y)^q w(y) \,dy \bigg)^{\frac{1}{q}} \, \lim_{t \rightarrow \infty} \bigg(\sup_{t \le \tau} u(\tau) \bigg( \int_0^{\tau} v(s)^{1-p'}\,ds \bigg)^{\frac{1}{p'}}\bigg). \label{eq.II}
		\end{align}
		
		Combining \eqref{eq.I} and \eqref{eq.II}, we arrive at
		\begin{align*}
		\sup_{h \ge  0} \frac{\bigg( \int_0^{\infty} \bigg( \int_0^x \bigg( \sup_{t \le \tau}u(\tau) \int_0^{\tau} h(y)\,dy \bigg) a(t)\,dt \bigg)^q w(x)\,dx \bigg)^{\frac{1}{q}}}{\bigg( \int_0^{\infty} h(s)^pv(s)\,ds\bigg)^{\frac{1}{p}}} & \\
		& \hspace{-7cm} \approx \, \sup_{t \in (0,\infty)} \bigg( \int_0^t v(x)^{1-p'} \, \bigg( \int_x^t \bigg( \sup_{s \le \tau}u(\tau)\bigg) a(s)\,ds \bigg)^{p'} \, dx \bigg)^{\frac{1}{p'}} \, \bigg( \int_t^{\infty} w(y) \,dy \bigg)^{\frac{1}{q}} \notag \\
		& \hspace{-6.5cm} + \sup_{t \in (0,\infty)} \bigg( \int_0^t v(x)^{1-p'}  \, dx \bigg)^{\frac{1}{p'}} \, \bigg( \int_t^{\infty} \, \bigg( \int_t^y \bigg( \sup_{s \le \tau}u(\tau)\bigg) a(s)\,ds \bigg)^{q} w(y) \,dy \bigg)^{\frac{1}{q}} \\
		& \hspace{-6.5cm} + \, \sup_{x \in (0,\infty)} \bigg( \int_{[x,\infty)} \, d \, \bigg( - \sup_{t \le \tau} u(\tau)^{p'} \bigg( \int_0^{\tau} v(s)^{1-p'}\,ds \bigg) \bigg) \bigg)^{\frac{1}{p'}} \, \bigg( \int_0^x A(y)^q w(y) \,dy \bigg)^{\frac{1}{q}} \notag \\
		& \hspace{-6.5cm} + \sup_{x \in (0,\infty)} \bigg( \int_{(0,x]} \,  A(t)^{p'} \, d \, \bigg( - \sup_{t \le \tau} u(\tau)^{p'} \bigg( \int_0^{\tau} v(s)^{1-p'}\,ds \bigg) \bigg) \bigg)^{\frac{1}{p'}} \, \bigg( \int_x^{\infty} w(y) \,dy \bigg)^{\frac{1}{q}} \notag \\
		& \hspace{-6.5cm} + \bigg( \int_0^{\infty} A(y)^q w(y) \,dy \bigg)^{\frac{1}{q}} \, \lim_{t \rightarrow \infty} \bigg(\sup_{t \le \tau} u(\tau) \bigg( \int_0^{\tau} v(s)^{1-p'}\,ds \bigg)^{\frac{1}{p'}}\bigg).
		\end{align*}
		
		{\rm (ii)} Let now $q < p$. By \cite[Theorem 1.2]{Oinar}, we obtain that
		\begin{align}
		\sup_{g \ge  0} \frac{ D }{\bigg( \int_0^{\infty} g^{q'}w^{1-q'}\bigg)^{\frac{1}{q'}}} & \notag \\
		& \hspace{-3cm} = \, \sup_{g \ge  0} \frac{\bigg( \int_0^{\infty} \bigg( \int_t^{\infty} g(x) \int_t^x \bigg( \sup_{s \le y}u(y)\bigg) a(s)\,ds \,dx \bigg)^{p'} \, v(t)^{1-p'} \,dt \bigg)^{\frac{1}{p'}} }{\bigg( \int_0^{\infty} g^{q'}w^{1-q'}\bigg)^{\frac{1}{q'}}} \notag \\
		& \hspace{-3cm} \approx \, \bigg( \int_0^{\infty}  \bigg( \int_0^t v(x)^{1-p'} \bigg)^{\frac{r}{q'}} v(t)^{1-p'} \, \bigg( \int_t^{\infty} \bigg( \int_t^z \bigg( \sup_{s \le y}u(y)\bigg) a(s)\,ds \bigg)^{q} w(z)\, dz \bigg)^{\frac{r}{q}} \, dt \bigg)^{\frac{1}{r}} \notag \\
		& \hspace{-2.5cm} + \bigg( \int_0^{\infty} \bigg( \int_0^t v(x)^{1-p'} \bigg( \int_x^t \bigg( \sup_{s \le y}u(y)\bigg) a(s)\,ds \bigg)^{p'} \, dx \bigg)^{\frac{r}{p'}} \bigg( \int_t^{\infty} w(s) \,ds \bigg)^{\frac{r}{p}} w(t)\,dt \bigg)^{\frac{1}{r}}. \label{eq.I.0}
		\end{align}
		
		By \cite[Theorem 2, p. 48]{mazya}, we have that 
		\begin{align*}
		\sup_{g \ge  0} \frac{ E_1 }{\bigg( \int_0^{\infty} g^{q'}w^{1-q'}\bigg)^{\frac{1}{q'}}} & \\
		& \hspace{-3cm} = \sup_{g \ge  0} \frac{\bigg( \int_0^{\infty} \bigg( \int_0^t  g(x)A(x)\,dx \bigg)^{p'} \,d \, \bigg( - \bigg(\sup_{t \le \tau} u(\tau)^{p'} \, \bigg( \int_0^{\tau} v(s)^{1-p'}\,ds \bigg) \bigg) \bigg)  \bigg)^{\frac{1}{p'}}}{\bigg( \int_0^{\infty} g^{q'}w^{1-q'}\bigg)^{\frac{1}{q'}}} \\
		& \hspace{-3cm} \approx \bigg( \int_0^{\infty} \bigg( \int_{[x,\infty)} \, d \, \bigg( - \bigg(\sup_{t \le \tau} u(\tau)^{p'} \, \bigg( \int_0^{\tau} v(s)^{1-p'}\,ds \bigg) \bigg) \bigg) \bigg)^{\frac{r}{p'}} \, \bigg( \int_0^x A(y)^q w(y) \,dy \bigg)^{\frac{r}{p}} A(x)^q w(x) \,dx\bigg)^{\frac{1}{r}}
		\end{align*}
		and
		\begin{align*}
		\sup_{g \ge  0} \frac{ E_2 }{\bigg( \int_0^{\infty} g^{q'}w^{1-q'}\bigg)^{\frac{1}{q'}}} & \\
		& \hspace{-3cm} = \sup_{g \ge  0} \frac{\bigg( \int_0^{\infty} \bigg( \int_t^{\infty} g(x)\,dx \bigg)^{p'} A(t)^{p'}
		\,d \, \bigg( - \bigg(\sup_{t \le \tau} u(\tau)^{p'} \, \bigg( \int_0^{\tau} v(s)^{1-p'}\,ds \bigg) \bigg) \bigg)  \bigg)^{\frac{1}{p'}}}{\bigg( \int_0^{\infty} g^{q'}w^{1-q'}\bigg)^{\frac{1}{q'}}} \\
		& \hspace{-3cm} \approx \bigg( \int_0^{\infty} \bigg( \int_{(0,x]} \,  A(t)^{p'} \, d \, \bigg( - \bigg(\sup_{t \le \tau} u(\tau)^{p'} \, \bigg( \int_0^{\tau} v(s)^{1-p'}\,ds \bigg) \bigg) \bigg)  \bigg)^{\frac{r}{p'}} \, \bigg( \int_x^{\infty} w(y) \,dy \bigg)^{\frac{r}{p}} w(x)\,dx \bigg)^{\frac{1}{r}}.
		\end{align*}
		
		Consequently, we arrive at
		\begin{align}
		\sup_{g \ge  0} \frac{E}{\bigg( \int_0^{\infty} g^{q'}w^{1-q'}\bigg)^{\frac{1}{q'}}} & \notag \\
		& \hspace{-3cm} \approx \, \bigg( \int_0^{\infty} \bigg( \int_{[x,\infty)} \, d \, \bigg( - \bigg(\sup_{t \le \tau} u(\tau)^{p'} \, \bigg( \int_0^{\tau} v(s)^{1-p'}\,ds \bigg) \bigg) \bigg) \bigg)^{\frac{r}{p'}} 
		\, \bigg( \int_0^x A(y)^q w(y) \,dy \bigg)^{\frac{r}{p}} A(x)^q w(x) \,dx\bigg)^{\frac{1}{r}} \notag \\
		& \hspace{-2.5cm} + \bigg( \int_0^{\infty} \bigg( \int_{(0,x]} \,  A(t)^{p'} \, d \, \bigg( - \bigg(\sup_{t \le \tau} u(\tau)^{p'} \, \bigg( \int_0^{\tau} v(s)^{1-p'}\,ds \bigg) \bigg) \bigg)  \bigg)^{\frac{r}{p'}} \, \bigg( \int_x^{\infty} w(y) \,dy \bigg)^{\frac{r}{p}} w(x)\,dx \bigg)^{\frac{1}{r}} \notag \\
		& \hspace{-2.5cm} + \bigg( \int_0^{\infty} A(y)^q w(y) \,dy \bigg)^{\frac{1}{q}} \, \lim_{t \rightarrow \infty} \bigg(\sup_{t \le \tau} u(\tau) \bigg( \int_0^{\tau} v(s)^{1-p'}\,ds \bigg)^{\frac{1}{p'}}\bigg). \label{eq.II.0}		
		\end{align}
		
		Combining \eqref{eq.I.0} and \eqref{eq.II.0}, we arrive at
		\begin{align*}
		\sup_{h \ge  0} \frac{\bigg( \int_0^{\infty} \bigg( \int_0^x \bigg( \sup_{t \le \tau}u(\tau) \int_0^{\tau} h(y)\,dy \bigg) a(t)\,dt \bigg)^q w(x)\,dx \bigg)^{\frac{1}{q}}}{\bigg( \int_0^{\infty} h(s)^pv(s)\,ds\bigg)^{\frac{1}{p}}} & \\
		& \hspace{-7cm} \approx \,  \bigg( \int_0^{\infty}  \bigg( \int_0^t v(x)^{1-p'} \bigg)^{\frac{r}{q'}} v(t)^{1-p'} \, \bigg( \int_t^{\infty} \bigg( \int_t^z \bigg( \sup_{s \le y}u(y)\bigg) a(s)\,ds \bigg)^{q} w(z)\, dz \bigg)^{\frac{r}{q}} \, dt \bigg)^{\frac{1}{r}} \notag \\
		& \hspace{-6.5cm} + \bigg( \int_0^{\infty} \bigg( \int_0^t v(x)^{1-p'} \bigg( \int_x^t \bigg( \sup_{s \le y}u(y)\bigg) a(s)\,ds \bigg)^{p'} \, dx \bigg)^{\frac{r}{p'}} \bigg( \int_t^{\infty} w(s) \,ds \bigg)^{\frac{r}{p}} w(t)\,dt \bigg)^{\frac{1}{r}} \notag \\
		& \hspace{-6.5cm} + \, \bigg( \int_0^{\infty} \bigg( \int_{[x,\infty)} \, d \, \bigg( - \bigg(\sup_{t \le \tau} u(\tau)^{p'} \, \bigg( \int_0^{\tau} v(s)^{1-p'}\,ds \bigg) \bigg) \bigg) \bigg)^{\frac{r}{p'}} 
		\, \bigg( \int_0^x A(y)^q w(y) \,dy \bigg)^{\frac{r}{p}} A(x)^q w(x) \,dx\bigg)^{\frac{1}{r}} \notag \\
		& \hspace{-6.5cm} + \bigg( \int_0^{\infty} \bigg( \int_{(0,x]} \,  A(t)^{p'} \, d \, \bigg( - \bigg(\sup_{t \le \tau} u(\tau)^{p'} \, \bigg( \int_0^{\tau} v(s)^{1-p'}\,ds \bigg) \bigg) \bigg)  \bigg)^{\frac{r}{p'}} \, \bigg( \int_x^{\infty} w(y) \,dy \bigg)^{\frac{r}{p}} w(x)\,dx \bigg)^{\frac{1}{r}} \notag \\
		& \hspace{-6.5cm} + \bigg( \int_0^{\infty} A(y)^q w(y) \,dy \bigg)^{\frac{1}{q}} \, \lim_{t \rightarrow \infty} \bigg(\sup_{t \le \tau} u(\tau) \bigg( \int_0^{\tau} v(s)^{1-p'}\,ds \bigg)^{\frac{1}{p'}}\bigg). 		
		\end{align*}
		
		The proof is completed. 
	\end{proof}

	\begin{thm}\label{aux.thm.1.1}
	Let $1 < p,\, q < \infty$ and $b \in \W\I$ be such that $b(t) > 0$ for a.e. $t\in (0,\infty)$. Assume that $u \in \W\I \cap C\I$ and $a,\,v,\,w \in \W\I$. Moreover, assume that
	$$
	0 < \int_0^x v(t)^{1-p'}\,dt < \infty \qquad \mbox{for all} \quad x > 0.
	$$ 
	
	{\rm (i)} If $p \le q$, then
	\begin{align*}
	\sup_{h \ge  0} \frac{\bigg( \int_0^{\infty} \bigg( \int_0^x \bigg( \sup_{t \le \tau} \frac{u(\tau)}{B(\tau)} \int_0^{\tau} h(y)b(y)\,dy \bigg) a(t)\,dt \bigg)^q w(x)\,dx \bigg)^{\frac{1}{q}}}{\bigg( \int_0^{\infty} h(s)^pv(s)\,ds\bigg)^{\frac{1}{p}}} & \\
	& \hspace{-7cm} \approx \, \sup_{t \in (0,\infty)} \bigg( \int_0^t b(x)^{p'}v(x)^{1-p'} \, \bigg( \int_x^t \bigg( \sup_{s \le \tau}\frac{u(\tau)}{B(\tau)}\bigg) a(s)\,ds \bigg)^{p'} \, dx \bigg)^{\frac{1}{p'}} \, \bigg( \int_t^{\infty} w(y) \,dy \bigg)^{\frac{1}{q}} \notag \\
	& \hspace{-6.5cm} + \sup_{t \in (0,\infty)} \bigg( \int_0^t b(x)^{p'}v(x)^{1-p'}  \, dx \bigg)^{\frac{1}{p'}} \, \bigg( \int_t^{\infty} \, \bigg( \int_t^y \bigg( \sup_{s \le \tau}\frac{u(\tau)}{B(\tau)}\bigg) a(s)\,ds \bigg)^{q} w(y) \,dy \bigg)^{\frac{1}{q}} \\
	& \hspace{-6.5cm} + \, \sup_{x \in (0,\infty)} \bigg( \int_{[x,\infty)} \, d \, \bigg( - \sup_{t \le \tau} \bigg(\frac{u(\tau)}{B(\tau)}\bigg)^{p'} \bigg( \int_0^{\tau} b(s)^{p'} v(s)^{1-p'}\,ds \bigg) \bigg) \bigg)^{\frac{1}{p'}} \, \bigg( \int_0^x A(y)^q w(y) \,dy \bigg)^{\frac{1}{q}} \notag \\
	& \hspace{-6.5cm} + \sup_{x \in (0,\infty)} \bigg( \int_{(0,x]} \,  A(t)^{p'} \, d \, \bigg( - \sup_{t \le \tau} \bigg(\frac{u(\tau)}{B(\tau)}\bigg)^{p'} \bigg( \int_0^{\tau} b(s)^{p'}v(s)^{1-p'}\,ds \bigg) \bigg) \bigg)^{\frac{1}{p'}} \, \bigg( \int_x^{\infty} w(y) \,dy \bigg)^{\frac{1}{q}} \notag \\
	& \hspace{-6.5cm} + \bigg( \int_0^{\infty} A(y)^q w(y) \,dy \bigg)^{\frac{1}{q}} \, \lim_{t \rightarrow \infty} \bigg(\sup_{t \le \tau} \frac{u(\tau)}{B(\tau)} \bigg( \int_0^{\tau} b(s)^{p'}v(s)^{1-p'}\,ds \bigg)^{\frac{1}{p'}}\bigg);
	\end{align*}
	
	{\rm (ii)} If $q < p$, then
	\begin{align*}
	\sup_{h \ge  0} \frac{\bigg( \int_0^{\infty} \bigg( \int_0^x \bigg( \sup_{t \le \tau}\frac{u(\tau)}{B(\tau)} \int_0^{\tau} h(y)b(y)\,dy \bigg) a(t)\,dt \bigg)^q w(x)\,dx \bigg)^{\frac{1}{q}}}{\bigg( \int_0^{\infty} h(s)^pv(s)\,ds\bigg)^{\frac{1}{p}}} & \\
	& \hspace{-7cm} \approx \, \bigg( \int_0^{\infty}  \bigg( \int_0^t b(x)^{p'}v(x)^{1-p'} \,dx \bigg)^{\frac{r}{q'}} b(t)^{p'}v(t)^{1-p'} \, \bigg( \int_t^{\infty} \bigg( \int_t^z \bigg( \sup_{s \le y}\frac{u(y)}{B(y)}\bigg) a(s)\,ds \bigg)^{q} w(z)\, dz \bigg)^{\frac{r}{q}} \, dt \bigg)^{\frac{1}{r}} \notag \\
	& \hspace{-6.5cm} + \bigg( \int_0^{\infty} \bigg( \int_0^t b(x)^{p'}v(x)^{1-p'} \bigg( \int_x^t \bigg( \sup_{s \le y}\frac{u(y)}{B(y)}\bigg) a(s)\,ds \bigg)^{p'} \, dx \bigg)^{\frac{r}{p'}} \bigg( \int_z^{\infty} w(s) \,ds \bigg)^{\frac{r}{p}} w(t)\,dt \bigg)^{\frac{1}{r}} \\
	& \hspace{-6.5cm} + \, \bigg( \int_0^{\infty} \bigg( \int_{[x,\infty)} \, d \, \bigg( - \bigg(\sup_{t \le \tau} \bigg(\frac{u(\tau)}{B(\tau)}\bigg)^{p'} \, \bigg( \int_0^{\tau} b(s)^{p'}v(s)^{1-p'}\,ds \bigg) \bigg) \bigg) \bigg)^{\frac{r}{p'}} 
	\, \bigg( \int_0^x A(y)^q w(y) \,dy \bigg)^{\frac{r}{p}} A(x)^q w(x) \,dx\bigg)^{\frac{1}{r}} \notag \\
	& \hspace{-6.5cm} + \bigg( \int_0^{\infty} \bigg( \int_{(0,x]} \,  A(t)^{p'} \, d \, \bigg( - \bigg(\sup_{t \le \tau} \bigg(\frac{u(\tau)}{B(\tau)}\bigg)^{p'} \, \bigg( \int_0^{\tau} b(s)^{p'}v(s)^{1-p'}\,ds \bigg) \bigg) \bigg)  \bigg)^{\frac{r}{p'}} \, \bigg( \int_x^{\infty} w(y) \,dy \bigg)^{\frac{r}{p}} w(x)\,dx \bigg)^{\frac{1}{r}} \notag \\
	& \hspace{-6.5cm} + \bigg( \int_0^{\infty} A(y)^q w(y) \,dy \bigg)^{\frac{1}{q}} \, \lim_{t \rightarrow \infty} \bigg(\sup_{t \le \tau} \frac{u(\tau)}{B(\tau)} \bigg( \int_0^{\tau} b(s)^{p'}v(s)^{1-p'}\,ds \bigg)^{\frac{1}{p'}}\bigg). 		
	\end{align*}
\end{thm}

\begin{proof}
The statement follows by Theorem \ref{aux.thm.1} at once if we note that
\begin{align*}
\sup_{h \ge  0} \frac{\bigg( \int_0^{\infty} \bigg( \int_0^x \bigg( \sup_{t \le \tau} \frac{u(\tau)}{B(\tau)} \int_0^{\tau} h(y)b(y)\,dy \bigg) a(t)\,dt \bigg)^q w(x)\,dx \bigg)^{\frac{1}{q}}}{\bigg( \int_0^{\infty} h(s)^pv(s)\,ds\bigg)^{\frac{1}{p}}} & \\
& \hspace{-7cm} = \sup_{h \ge  0} \frac{\bigg( \int_0^{\infty} \bigg( \int_0^x \bigg( \sup_{t \le \tau} \frac{u(\tau)}{B(\tau)} \int_0^{\tau} h(y)\,dy \bigg) a(t)\,dt \bigg)^q w(x)\,dx \bigg)^{\frac{1}{q}}}{\bigg( \int_0^{\infty} h(s)^p b(s)^{-p}v(s)\,ds\bigg)^{\frac{1}{p}}}.
\end{align*}

\end{proof}

	\begin{thm}\label{aux.thm.2}
		Let $1 < p,\, q < \infty$ and $b \in \W\I$ be such that $b(t) > 0$ for a.e. $t\in (0,\infty)$. Assume that $u \in \W\I \cap C\I$ and $a,\,v,\,w \in \W\I$. Moreover, assume that
		$$
		0 < \int_x^{\infty} v(t)^{1-p'}\,dt < \infty \qquad \mbox{for all} \quad x > 0.
		$$ 
		Denote by
		\begin{align*}
		\psi (x) & : = \bigg( \int_x^{\infty} b(t)^{p'}	v^{1-{p}^{\prime}}(t)\,dt\bigg)^{- \frac{p^{\prime}}{p^{\prime} + 1}} b(x)^{p'} v^{1-{p}^{\prime}}(x) \\
		\intertext{and}
		\Psi(x) & : = \bigg( \int_x^{\infty} b(t)^{p'} v^{1-{p}^{\prime}}(t)\,dt\bigg)^{\frac{1}{p^{\prime} + 1}}.
		\end{align*}
		
		{\rm (i)} If $p \le q$, then
        \begin{align*}
        \sup_{h \ge  0} \frac{\bigg( \int_0^{\infty} \bigg( \int_0^x \bigg( \sup_{t \le \tau}\frac{u(\tau)}{B(\tau)} \int_{\tau}^{\infty} h(y)b(y)\,dy \bigg) a(t)\,dt \bigg)^q w(x)\,dx \bigg)^{\frac{1}{q}}}{\bigg( \int_0^{\infty} h(s)^pv(s)\,ds\bigg)^{\frac{1}{p}}} & \\
        & \hspace{-7cm} \approx \, \sup_{t \in (0,\infty)} \bigg( \int_0^t \Psi(x)^{-p'} \psi(x) \, \bigg( \int_x^t \bigg( \sup_{s \le \tau} \frac{u(\tau)}{B(\tau)} \Psi(\tau)^2 \bigg) a(s)\,ds \bigg)^{p'} \, dx \bigg)^{\frac{1}{p'}} \, \bigg( \int_t^{\infty} w(y) \,dy \bigg)^{\frac{1}{q}} \notag \\
        & \hspace{-6.5cm} + \, \sup_{t \in (0,\infty)} \bigg( \int_0^t \Psi(x)^{-p'} \psi(x)  \, dx \bigg)^{\frac{1}{p'}} \, \bigg( \int_t^{\infty} \, \bigg( \int_t^y \bigg( \sup_{s \le \tau} \frac{u(\tau)}{B(\tau)} \Psi(\tau)^2\bigg) a(s)\,ds \bigg)^{q} w(y) \,dy \bigg)^{\frac{1}{q}} \\
        & \hspace{-6.5cm} + \, \sup_{x \in (0,\infty)} \bigg( \int_{[x,\infty)} \, d \, \bigg( - \sup_{t \le \tau} \bigg(\frac{u(\tau)}{B(\tau)}\bigg)^{p'} \Psi(\tau)^{2p'} \bigg( \int_0^{\tau} \Psi(s)^{-p'} \psi(s)\,ds \bigg) \bigg) \bigg)^{\frac{1}{p'}} \, \bigg( \int_0^x A(y)^q w(y) \,dy \bigg)^{\frac{1}{q}} \notag \\
        & \hspace{-6.5cm} + \,\sup_{x \in (0,\infty)} \bigg( \int_{(0,x]} \,  A(t)^{p'} \, d \, \bigg( - \sup_{t \le \tau} \bigg(\frac{u(\tau)}{B(\tau)}\bigg)^{p'} \Psi(\tau)^{2p'} \bigg( \int_0^{\tau} \Psi(s)^{-p'} \psi(s)\,ds \bigg) \bigg) \bigg)^{\frac{1}{p'}} \, \bigg( \int_x^{\infty} w(y) \,dy \bigg)^{\frac{1}{q}} \notag \\
        & \hspace{-6.5cm} + \,\bigg( \int_0^{\infty} A(y)^q w(y) \,dy \bigg)^{\frac{1}{q}} \, \lim_{t \rightarrow \infty} \bigg(\sup_{t \le \tau} \frac{u(\tau)}{B(\tau)} \Psi(\tau)^2 \bigg( \int_0^{\tau} \Psi(s)^{-p'} \psi(s)\,ds \bigg)^{\frac{1}{p'}}\bigg) \\
        & \hspace{-6.5cm} + \, \bigg( \int_0^{\infty} \psi(s)\,ds\bigg)^{-\frac{1}{p}} \bigg( \int_0^{\infty} \bigg( \int_0^x \bigg( \sup_{t \le \tau} \frac{u(\tau)}{B(\tau)} \Psi(\tau)^2 \bigg) a(t)\,dt \bigg)^q w(x)\,dx \bigg)^{\frac{1}{q}};
        \end{align*}
		
		{\rm (ii)} If $q < p$, then
		\begin{align*}
		\sup_{h \ge  0} \frac{\bigg( \int_0^{\infty} \bigg( \int_0^x \bigg( \sup_{t \le \tau} \frac{u(\tau)}{B(\tau)} \int_{\tau}^{\infty} h(y)b(y)\,dy \bigg) a(t)\,dt \bigg)^q w(x)\,dx \bigg)^{\frac{1}{q}}}{\bigg( \int_0^{\infty} h(s)^pv(s)\,ds\bigg)^{\frac{1}{p}}} & \\
		& \hspace{-8cm} \approx \, \bigg( \int_0^{\infty}  \bigg( \int_0^t \Psi(x)^{-p'} \psi(x) \,dx \bigg)^{\frac{r}{q'}} \Psi(t)^{-p'} \psi(t) \, \bigg( \int_t^{\infty} \bigg( \int_t^z \bigg( \sup_{s \le y} \frac{u(y)}{B(y)}\Psi(y)^2\bigg) a(s)\,ds \bigg)^{q} w(z)\, dz \bigg)^{\frac{r}{q}} \, dt \bigg)^{\frac{1}{r}} \notag \\
		& \hspace{-7.5cm} + \bigg( \int_0^{\infty} \bigg( \int_0^t \Psi(x)^{-p'} \psi(x) \bigg( \int_x^t \bigg( \sup_{s \le y}\frac{u(y)}{B(y)}\Psi(y)^2\bigg) a(s)\,ds \bigg)^{p'} \, dx \bigg)^{\frac{r}{p'}} \bigg( \int_z^{\infty} w(s) \,ds \bigg)^{\frac{r}{p}} w(t)\,dt \bigg)^{\frac{1}{r}} \\
		& \hspace{-7.5cm} + \, \bigg( \int_0^{\infty} \bigg( \int_{[x,\infty)} \, d \, \bigg( - \bigg(\sup_{t \le \tau} \bigg(\frac{u(\tau)}{B(\tau)}\bigg)^{p'} \Psi(\tau)^{2p'} \, \bigg( \int_0^{\tau} \Psi(s)^{-p'} \psi(s)\,ds \bigg) \bigg) \bigg) \bigg)^{\frac{r}{p'}} \, \bigg( \int_0^x A(y)^q w(y) \,dy \bigg)^{\frac{r}{p}} A(x)^q w(x) \,dx\bigg)^{\frac{1}{r}} \notag \\
		& \hspace{-7.5cm} + \bigg( \int_0^{\infty} \bigg( \int_{(0,x]} \,  A(t)^{p'} \, d \, \bigg( - \bigg(\sup_{t \le \tau} \bigg(\frac{u(\tau)}{B(\tau)}\bigg)^{p'} \Psi(\tau)^{2p'} \, \bigg( \int_0^{\tau} \Psi(s)^{-p'} \psi(s)\,ds \bigg) \bigg) \bigg)  \bigg)^{\frac{r}{p'}} \, \bigg( \int_x^{\infty} w(y) \,dy \bigg)^{\frac{r}{p}} w(x)\,dx \bigg)^{\frac{1}{r}} \notag \\
		& \hspace{-7.5cm} + \bigg( \int_0^{\infty} A(y)^q w(y) \,dy \bigg)^{\frac{1}{q}} \, \lim_{t \rightarrow \infty} \bigg(\sup_{t \le \tau} \frac{u(\tau)}{B(\tau)} \Psi(\tau)^2 \,\bigg( \int_0^{\tau} \Psi(s)^{-p'} \psi(s)\,ds \bigg)^{\frac{1}{p'}}\bigg) \\
		& \hspace{-7.5cm} + \, \bigg( \int_0^{\infty} \psi(s)\,ds\bigg)^{-\frac{1}{p}} \bigg( \int_0^{\infty} \bigg( \int_0^x \bigg( \sup_{t \le \tau}\frac{u(\tau)}{B(\tau)} \Psi(\tau)^2 \bigg) a(t)\,dt \bigg)^q w(x)\,dx \bigg)^{\frac{1}{q}}.
		\end{align*}       
		
	\end{thm}
	
	\begin{proof}
	By \cite[Corollary 3.5]{GogMusIHI}, we have that
	\begin{align*}
	\sup_{h \ge  0} \frac{\bigg( \int_0^{\infty} \bigg( \int_0^x \bigg( \sup_{t \le \tau}\frac{u(\tau)}{B(\tau)} \int_{\tau}^{\infty} h(y)b(y)\,dy \bigg) a(t)\,dt \bigg)^q w(x)\,dx \bigg)^{\frac{1}{q}}}{\bigg( \int_0^{\infty} h(s)^pv(s)\,ds\bigg)^{\frac{1}{p}}} & \\
	& \hspace{-7cm} = \, \sup_{h \ge  0} \frac{\bigg( \int_0^{\infty} \bigg( \int_0^x \bigg( \sup_{t \le \tau}\frac{u(\tau)}{B(\tau)} \int_{\tau}^{\infty} h(y)\,dy \bigg) a(t)\,dt \bigg)^q w(x)\,dx \bigg)^{\frac{1}{q}}}{\bigg( \int_0^{\infty} h(s)^p b(s)^{-p}v(s)\,ds\bigg)^{\frac{1}{p}}} \\
	& \hspace{-7cm} \approx \, \sup_{h \ge  0} \frac{\bigg( \int_0^{\infty} \bigg( \int_0^x \bigg( \sup_{t \le \tau} \frac{u(\tau)}{B(\tau)}\Psi(\tau)^2 \int_0^{\tau} h(y)\,dy \bigg) a(t)\,dt \bigg)^q w(x)\,dx \bigg)^{\frac{1}{q}}}{\bigg( \int_0^{\infty} h(s)^p \Psi(s)^p \psi(s)^{1-p}\,ds\bigg)^{\frac{1}{p}}} \\
	& \hspace{-6.5cm} + \, \frac{\bigg( \int_0^{\infty} \bigg( \int_0^x \bigg( \sup_{t \le \tau} \frac{u(\tau)}{B(\tau)} \Psi(\tau)^2 \bigg) a(t)\,dt \bigg)^q w(x)\,dx \bigg)^{\frac{1}{q}}}{\bigg( \int_0^{\infty} \psi(s)\,ds\bigg)^{\frac{1}{p}}}.
	\end{align*}
    
    {\rm (i)} Let $p \le q$. By Theorem \ref{aux.thm.1}, (i),  we get that	
    \begin{align*}
    \sup_{h \ge  0} \frac{\bigg( \int_0^{\infty} \bigg( \int_0^x \bigg( \sup_{t \le \tau}\frac{u(\tau)}{B(\tau)} \int_{\tau}^{\infty} h(y)b(y)\,dy \bigg) a(t)\,dt \bigg)^q w(x)\,dx \bigg)^{\frac{1}{q}}}{\bigg( \int_0^{\infty} h(s)^pv(s)\,ds\bigg)^{\frac{1}{p}}} & \\
    & \hspace{-7cm} \approx \, \sup_{t \in (0,\infty)} \bigg( \int_0^t \Psi(x)^{-p'} \psi(x) \, \bigg( \int_x^t \bigg( \sup_{s \le \tau} \frac{u(\tau)}{B(\tau)} \Psi(\tau)^2 \bigg) a(s)\,ds \bigg)^{p'} \, dx \bigg)^{\frac{1}{p'}} \, \bigg( \int_t^{\infty} w(y) \,dy \bigg)^{\frac{1}{q}} \notag \\
    & \hspace{-6.5cm} + \, \sup_{t \in (0,\infty)} \bigg( \int_0^t \Psi(x)^{-p'} \psi(x)  \, dx \bigg)^{\frac{1}{p'}} \, \bigg( \int_t^{\infty} \, \bigg( \int_t^y \bigg( \sup_{s \le \tau} \frac{u(\tau)}{B(\tau)} \Psi(\tau)^2\bigg) a(s)\,ds \bigg)^{q} w(y) \,dy \bigg)^{\frac{1}{q}} \\
    & \hspace{-6.5cm} + \, \sup_{x \in (0,\infty)} \bigg( \int_{[x,\infty)} \, d \, \bigg( - \sup_{t \le \tau} \bigg(\frac{u(\tau)}{B(\tau)}\bigg)^{p'} \Psi(\tau)^{2p'} \bigg( \int_0^{\tau} \Psi(s)^{-p'} \psi(s)\,ds \bigg) \bigg) \bigg)^{\frac{1}{p'}} \, \bigg( \int_0^x A(y)^q w(y) \,dy \bigg)^{\frac{1}{q}} \notag \\
    & \hspace{-6.5cm} + \,\sup_{x \in (0,\infty)} \bigg( \int_{(0,x]} \,  A(t)^{p'} \, d \, \bigg( - \sup_{t \le \tau} \bigg(\frac{u(\tau)}{B(\tau)}\bigg)^{p'} \Psi(\tau)^{2p'} \bigg( \int_0^{\tau} \Psi(s)^{-p'} \psi(s)\,ds \bigg) \bigg) \bigg)^{\frac{1}{p'}} \, \bigg( \int_x^{\infty} w(y) \,dy \bigg)^{\frac{1}{q}} \notag \\
    & \hspace{-6.5cm} + \,\bigg( \int_0^{\infty} A(y)^q w(y) \,dy \bigg)^{\frac{1}{q}} \, \lim_{t \rightarrow \infty} \bigg(\sup_{t \le \tau} \frac{u(\tau)}{B(\tau)} \Psi(\tau)^2 \bigg( \int_0^{\tau} \Psi(s)^{-p'} \psi(s)\,ds \bigg)^{\frac{1}{p'}}\bigg) \\
    & \hspace{-6.5cm} + \, \bigg( \int_0^{\infty} \psi(s)\,ds\bigg)^{-\frac{1}{p}} \bigg( \int_0^{\infty} \bigg( \int_0^x \bigg( \sup_{t \le \tau} \frac{u(\tau)}{B(\tau)} \Psi(\tau)^2 \bigg) a(t)\,dt \bigg)^q w(x)\,dx \bigg)^{\frac{1}{q}};
    \end{align*}
    
    {\rm (ii)} Let $q < p$. By Theorem \ref{aux.thm.1}, (ii),  we obtain that	
    \begin{align*}
    \sup_{h \ge  0} \frac{\bigg( \int_0^{\infty} \bigg( \int_0^x \bigg( \sup_{t \le \tau} \frac{u(\tau)}{B(\tau)} \int_{\tau}^{\infty} h(y)b(y)\,dy \bigg) a(t)\,dt \bigg)^q w(x)\,dx \bigg)^{\frac{1}{q}}}{\bigg( \int_0^{\infty} h(s)^pv(s)\,ds\bigg)^{\frac{1}{p}}} & \\
    & \hspace{-8cm} \approx \, \bigg( \int_0^{\infty}  \bigg( \int_0^t \Psi(x)^{-p'} \psi(x) \,dx \bigg)^{\frac{r}{q'}} \Psi(t)^{-p'} \psi(t) \, \bigg( \int_t^{\infty} \bigg( \int_t^z \bigg( \sup_{s \le y} \frac{u(y)}{B(y)}\Psi(y)^2\bigg) a(s)\,ds \bigg)^{q} w(z)\, dz \bigg)^{\frac{r}{q}} \, dt \bigg)^{\frac{1}{r}} \notag \\
    & \hspace{-7.5cm} + \bigg( \int_0^{\infty} \bigg( \int_0^t \Psi(x)^{-p'} \psi(x) \bigg( \int_x^t \bigg( \sup_{s \le y}\frac{u(y)}{B(y)}\Psi(y)^2\bigg) a(s)\,ds \bigg)^{p'} \, dx \bigg)^{\frac{r}{p'}} \bigg( \int_z^{\infty} w(s) \,ds \bigg)^{\frac{r}{p}} w(t)\,dt \bigg)^{\frac{1}{r}} \\
    & \hspace{-7.5cm} + \, \bigg( \int_0^{\infty} \bigg( \int_{[x,\infty)} \, d \, \bigg( - \bigg(\sup_{t \le \tau} \bigg(\frac{u(\tau)}{B(\tau)}\bigg)^{p'} \Psi(\tau)^{2p'} \, \bigg( \int_0^{\tau} \Psi(s)^{-p'} \psi(s)\,ds \bigg) \bigg) \bigg) \bigg)^{\frac{r}{p'}}  \, \bigg( \int_0^x A(y)^q w(y) \,dy \bigg)^{\frac{r}{p}} A(x)^q w(x) \,dx\bigg)^{\frac{1}{r}} \notag \\
    & \hspace{-7.5cm} + \bigg( \int_0^{\infty} \bigg( \int_{(0,x]} \,  A(t)^{p'} \, d \, \bigg( - \bigg(\sup_{t \le \tau} \bigg(\frac{u(\tau)}{B(\tau)}\bigg)^{p'} \Psi(\tau)^{2p'} \, \bigg( \int_0^{\tau} \Psi(s)^{-p'} \psi(s)\,ds \bigg) \bigg) \bigg)  \bigg)^{\frac{r}{p'}} \, \bigg( \int_x^{\infty} w(y) \,dy \bigg)^{\frac{r}{p}} w(x)\,dx \bigg)^{\frac{1}{r}} \notag \\
    & \hspace{-7.5cm} + \bigg( \int_0^{\infty} A(y)^q w(y) \,dy \bigg)^{\frac{1}{q}} \, \lim_{t \rightarrow \infty} \bigg(\sup_{t \le \tau} \frac{u(\tau)}{B(\tau)} \Psi(\tau)^2 \,\bigg( \int_0^{\tau} \Psi(s)^{-p'} \psi(s)\,ds \bigg)^{\frac{1}{p'}}\bigg) \\
    & \hspace{-7.5cm} + \, \bigg( \int_0^{\infty} \psi(s)\,ds\bigg)^{-\frac{1}{p}} \bigg( \int_0^{\infty} \bigg( \int_0^x \bigg( \sup_{t \le \tau}\frac{u(\tau)}{B(\tau)} \Psi(\tau)^2 \bigg) a(t)\,dt \bigg)^q w(x)\,dx \bigg)^{\frac{1}{q}}.
    \end{align*}       
    
    The proof is completed.
	\end{proof}

\
	
\section{The boundedness of $R_u$ from $L^p(v)$ into $\ces_q(w,a)$ on the cone of monotone non-increasing functions}\label{R}	

\

In this section we characterize the boundedness of $R_u$ from $L^p(v)$ into $\ces_q(w,a)$ on the cone of monotone non-increasing functions.
\begin{thm}\label{thm.R}
	Let $1 < p,\, q < \infty$. Assume that $u \in \W\I \cap C\I$ and $a,\,v,\,w \in \W\I$.
	
	{\rm (i)} If $p \le q$, then
	\begin{align*}
	\sup_{f \in \mp^{+,\dn} (0,\infty)} \frac{\bigg( \int_0^{\infty} \bigg( \int_0^x (R_u f) (t) a(t)\,dt \bigg)^q w(x)\,dx \bigg)^{\frac{1}{q}}}{\bigg( \int_0^{\infty} f(s)^pv(s)\,ds\bigg)^{\frac{1}{p}}} & \\
	& \hspace{-6cm} \approx \, \sup_{t \in (0,\infty)} \bigg( \int_0^t V(x)^{p'} v(x) \, \bigg( \int_x^t \bigg( \sup_{s \le \tau}u(\tau) V(\tau)^{-2}\bigg) a(s)\,ds \bigg)^{p'} \, dx \bigg)^{\frac{1}{p'}} \, \bigg( \int_t^{\infty} w(y) \,dy \bigg)^{\frac{1}{q}} \notag \\
	& \hspace{-5.5cm} + \sup_{t \in (0,\infty)} \bigg( \int_0^t V(x)^{p'} v(x)  \, dx \bigg)^{\frac{1}{p'}} \, \bigg( \int_t^{\infty} \, \bigg( \int_t^y \bigg( \sup_{s \le \tau}u(\tau) V(\tau)^{-2}\bigg) a(s)\,ds \bigg)^{q} w(y) \,dy \bigg)^{\frac{1}{q}} \\
	& \hspace{-5.5cm} + \, \sup_{x \in (0,\infty)} \bigg( \int_{[x,\infty)} \, d \, \bigg( - \sup_{t \le \tau} u(\tau)^{p'} V(\tau)^{-2p'} \bigg( \int_0^{\tau} V(s)^{p'} v(s)\,ds \bigg) \bigg) \bigg)^{\frac{1}{p'}} \, \bigg( \int_0^x A(y)^q w(y) \,dy \bigg)^{\frac{1}{q}} \notag \\
	& \hspace{-5.5cm} + \sup_{x \in (0,\infty)} \bigg( \int_{(0,x]} \,  A(t)^{p'} \, d \, \bigg( - \sup_{t \le \tau} u(\tau)^{p'} V(\tau)^{-2p'} \bigg( \int_0^{\tau} V(s)^{p'} v(s)\,ds \bigg) \bigg) \bigg)^{\frac{1}{p'}} \, \bigg( \int_x^{\infty} w(y) \,dy \bigg)^{\frac{1}{q}} \notag \\
	& \hspace{-5.5cm} + \bigg( \int_0^{\infty} A(y)^q w(y) \,dy \bigg)^{\frac{1}{q}} \, \lim_{t \rightarrow \infty} \bigg(\sup_{t \le \tau} u(\tau) V(\tau)^{-2} \bigg( \int_0^{\tau} V(s)^{p'} v(s)\,ds \bigg)^{\frac{1}{p'}}\bigg);
	\end{align*}
	
	{\rm (ii)} If $q < p$, then
	\begin{align*}
	\sup_{f \in \mp^{+,\dn} (0,\infty)} \frac{\bigg( \int_0^{\infty} \bigg( \int_0^x (R_u f) (t) a(t)\,dt \bigg)^q w(x)\,dx \bigg)^{\frac{1}{q}}}{\bigg( \int_0^{\infty} f(s)^p v(s)\,ds\bigg)^{\frac{1}{p}}} & \\
	& \hspace{-6cm} \approx \, \bigg( \int_0^{\infty}  \bigg( \int_0^t V(x)^{p'} v(x) \,dx \bigg)^{\frac{r}{q'}} V(t)^{p'} v(t) \, \bigg( \int_t^{\infty} \bigg( \int_t^z \bigg( \sup_{s \le y}u(y)V(y)^{-2}\bigg) a(s)\,ds \bigg)^{q} w(z)\, dz \bigg)^{\frac{r}{q}} \, dt \bigg)^{\frac{1}{r}} \notag \\
	& \hspace{-5.5cm} + \bigg( \int_0^{\infty} \bigg( \int_0^t V(x)^{p'} v(x) \bigg( \int_x^t \bigg( \sup_{s \le y}u(y)V(y)^{-2}\bigg) a(s)\,ds \bigg)^{p'} \, dx \bigg)^{\frac{r}{p'}} \bigg( \int_z^{\infty} w(s) \,ds \bigg)^{\frac{r}{p}} w(t)\,dt \bigg)^{\frac{1}{r}} \\
	& \hspace{-5.5cm} + \, \bigg( \int_0^{\infty} \bigg( \int_{[x,\infty)} \, d \, \bigg( - \bigg(\sup_{t \le \tau} u(\tau)^{p'}V(\tau)^{-2p'} \, \bigg( \int_0^{\tau} V(s)^{p'} v(s)\,ds \bigg) \bigg) \bigg) \bigg)^{\frac{r}{p'}} \, \bigg( \int_0^x A(y)^q w(y) \,dy \bigg)^{\frac{r}{p}} x^q w(x) \,dx\bigg)^{\frac{1}{r}} \notag \\
	& \hspace{-5.5cm} + \bigg( \int_0^{\infty} \bigg( \int_{(0,x]} \,  A(t)^{p'} \, d \, \bigg( - \bigg(\sup_{t \le \tau} u(\tau)^{p'}V(\tau)^{-2p'} \, \bigg( \int_0^{\tau} V(s)^{p'} v(s)\,ds \bigg) \bigg) \bigg)  \bigg)^{\frac{r}{p'}} \, \bigg( \int_x^{\infty} w(y) \,dy \bigg)^{\frac{r}{p}} w(x)\,dx \bigg)^{\frac{1}{r}} \notag \\
	& \hspace{-5.5cm} + \bigg( \int_0^{\infty} A(y)^q w(y) \,dy \bigg)^{\frac{1}{q}} \, \lim_{t \rightarrow \infty} \bigg(\sup_{t \le \tau} u(\tau)V(\tau)^{-2} \bigg( \int_0^{\tau} V(s)^{p'} v(s)\,ds \bigg)^{\frac{1}{p'}}\bigg). 		
		\end{align*}   
	
\end{thm}

\begin{proof}
	By \cite[Theorem 3.2]{GogStep} (cf. \cite[Theorem 2.3]{GogMusIHI}), we get that
	\begin{align*}
	\sup_{f \in \mp^{+,\dn} (0,\infty)} \frac{\bigg( \int_0^{\infty} \bigg( \int_0^x \bigg( \sup_{t \le \tau}
		u(\tau) f(\tau) \bigg) a(t)\,dt \bigg)^q w(x)\,dx \bigg)^{\frac{1}{q}}}{\bigg( \int_0^{\infty} f(s)^pv(s)\,ds\bigg)^{\frac{1}{p}}} & \\
	& \hspace{-7cm} \approx \sup_{h \ge  0} \frac{\bigg( \int_0^{\infty} \bigg( \int_0^x \bigg( \sup_{t \le \tau} u(\tau) V(\tau)^{-2} \int_0^{\tau} h(y)\,dy \bigg) a(t) \,dt \bigg)^q w(x)\,dx \bigg)^{\frac{1}{q}}}{\bigg( \int_0^{\infty} h(s)^p  V(s)^{-p} v(s)^{1-p}\,ds\bigg)^{\frac{1}{p}}}.
	\end{align*}
	
	By Theorem \ref{aux.thm.1}, we have that
	
	{\rm (i)} if $p \le q$, then
	\begin{align*}
    \sup_{h \ge  0} \frac{\bigg( \int_0^{\infty} \bigg( \int_0^x \bigg( \sup_{t \le \tau} u(\tau) V(\tau)^{-2} \int_0^{\tau} h(y)\,dy \bigg) a(t)\, dt \bigg)^q w(x)\,dx \bigg)^{\frac{1}{q}}}{\bigg( \int_0^{\infty} h(s)^p  V(s)^{-p} v(s)^{1-p}\,ds\bigg)^{\frac{1}{p}}} & \\
	& \hspace{-7cm} \approx \, \sup_{t \in (0,\infty)} \bigg( \int_0^t V(x)^{p'} v(x) \, \bigg( \int_x^t \bigg( \sup_{s \le \tau}u(\tau) V(\tau)^{-2}\bigg) a(s) \,ds \bigg)^{p'} \, dx \bigg)^{\frac{1}{p'}} \, \bigg( \int_t^{\infty} w(y) \,dy \bigg)^{\frac{1}{q}} \notag \\
	& \hspace{-6.5cm} + \sup_{t \in (0,\infty)} \bigg( \int_0^t V(x)^{p'} v(x)  \, dx \bigg)^{\frac{1}{p'}} \, \bigg( \int_t^{\infty} \, \bigg( \int_t^y \bigg( \sup_{s \le \tau}u(\tau) V(\tau)^{-2}\bigg) a(s)\,ds \bigg)^{q} w(y) \,dy \bigg)^{\frac{1}{q}} \\
	& \hspace{-6.5cm} + \, \sup_{x \in (0,\infty)} \bigg( \int_{[x,\infty)} \, d \, \bigg( - \sup_{t \le \tau} u(\tau)^{p'} V(\tau)^{-2p'} \bigg( \int_0^{\tau} V(s)^{p'} v(s)\,ds \bigg) \bigg) \bigg)^{\frac{1}{p'}} \, \bigg( \int_0^x A(y)^q w(y) \,dy \bigg)^{\frac{1}{q}} \notag \\
	& \hspace{-6.5cm} + \sup_{x \in (0,\infty)} \bigg( \int_{(0,x]} \,  A(t)^{p'} \, d \, \bigg( - \sup_{t \le \tau} u(\tau)^{p'} V(\tau)^{-2p'} \bigg( \int_0^{\tau} V(s)^{p'} v(s)\,ds \bigg) \bigg) \bigg)^{\frac{1}{p'}} \, \bigg( \int_x^{\infty} w(y) \,dy \bigg)^{\frac{1}{q}} \notag \\
	& \hspace{-6.5cm} + \bigg( \int_0^{\infty} A(y)^q w(y) \,dy \bigg)^{\frac{1}{q}} \, \lim_{t \rightarrow \infty} \bigg(\sup_{t \le \tau} u(\tau) V(\tau)^{-2} \bigg( \int_0^{\tau} V(s)^{p'} v(s)\,ds \bigg)^{\frac{1}{p'}}\bigg).		
	\end{align*}
	
	{\rm (ii)} if $q < p$, then
	\begin{align*}
	\sup_{h \ge  0} \frac{\bigg( \int_0^{\infty} \bigg( \int_0^x \bigg( \sup_{t \le \tau} u(\tau) V(\tau)^{-2} \int_0^{\tau} h(y)\,dy \bigg) a(t)\, dt \bigg)^q w(x)\,dx \bigg)^{\frac{1}{q}}}{\bigg( \int_0^{\infty} h(s)^p  V(s)^{-p} v(s)^{1-p}\,ds\bigg)^{\frac{1}{p}}} & \\
	& \hspace{-8cm} \approx \, \bigg( \int_0^{\infty}  \bigg( \int_0^t V(x)^{p'} v(x) \,dx \bigg)^{\frac{r}{q'}} V(t)^{p'} v(t) \, \bigg( \int_t^{\infty} \bigg( \int_t^z \bigg( \sup_{s \le y}u(y)V(y)^{-2}\bigg) a(s)\,ds \bigg)^{q} w(z)\, dz \bigg)^{\frac{r}{q}} \, dt \bigg)^{\frac{1}{r}} \notag \\
	& \hspace{-7.5cm} + \bigg( \int_0^{\infty} \bigg( \int_0^t V(x)^{p'} v(x) \bigg( \int_x^t \bigg( \sup_{s \le y}u(y)V(y)^{-2}\bigg) a(s)\,ds \bigg)^{p'} \, dx \bigg)^{\frac{r}{p'}} \bigg( \int_z^{\infty} w(s) \,ds \bigg)^{\frac{r}{p}} w(t)\,dt \bigg)^{\frac{1}{r}} \\
	& \hspace{-7.5cm} + \, \bigg( \int_0^{\infty} \bigg( \int_{[x,\infty)} \, d \, \bigg( - \bigg(\sup_{t \le \tau} u(\tau)^{p'}V(\tau)^{-2p'} \, \bigg( \int_0^{\tau} V(s)^{p'} v(s)\,ds \bigg) \bigg) \bigg) \bigg)^{\frac{r}{p'}} \, \bigg( \int_0^x A(y)^q w(y) \,dy \bigg)^{\frac{r}{p}} A(x)^q w(x) \,dx\bigg)^{\frac{1}{r}} \notag \\
	& \hspace{-7.5cm} + \bigg( \int_0^{\infty} \bigg( \int_{(0,x]} \,  A(t)^{p'} \, d \, \bigg( - \bigg(\sup_{t \le \tau} u(\tau)^{p'}V(\tau)^{-2p'} \, \bigg( \int_0^{\tau} V(s)^{p'} v(s)\,ds \bigg) \bigg) \bigg)  \bigg)^{\frac{r}{p'}} \, \bigg( \int_x^{\infty} w(y) \,dy \bigg)^{\frac{r}{p}} w(x)\,dx \bigg)^{\frac{1}{r}} \notag \\
	& \hspace{-7.5cm} + \bigg( \int_0^{\infty} A(y)^q w(y) \,dy \bigg)^{\frac{1}{q}} \, \lim_{t \rightarrow \infty} \bigg(\sup_{t \le \tau} u(\tau)V(\tau)^{-2} \bigg( \int_0^{\tau} V(s)^{p'} v(s)\,ds \bigg)^{\frac{1}{p'}}\bigg). 		
	\end{align*}        
	
\end{proof}

\

\section{The boundedness of $P_{u,b}$ from $L^p(v)$ into $\ces_{q}(w,a)$ on the cone of monotone non-increasing functions}\label{P}

\

In this section we characterize the boundedness of weighted Hardy operator $P_{u,b}$ from $L^p(v)$ into $\ces_q(w,a)$ on the cone of monotone non-increasing functions.
\begin{thm}\label{aux.thm.3}
Let $1 < p,\, q < \infty$ and $b \in \W\I$ be such that $b(t) > 0$ for a.e. $t\in (0,\infty)$. Assume that $u \in \W\I \cap C\I$ and $a,\,v,\,w \in \W\I$.

{\rm (i)} If $p \le q$, then
\begin{align*}
\sup_{f \in \mp^{+,\dn} (0,\infty)} \frac{\bigg( \int_0^{\infty} \bigg( \int_0^x  (P_{u,b} f)(t)  a(t)\,dt \bigg)^q w(x)\,dx \bigg)^{\frac{1}{q}}}{\bigg( \int_0^{\infty} f(s)^pv(s)\,ds\bigg)^{\frac{1}{p}}} & \\
& \hspace{-6cm} \approx \sup_{x \in (0,\infty)} \bigg( \int_0^x \bigg(\int_0^t a(y)u(y)\,dy\bigg)^q w(t)\,dt \bigg)^{\frac{1}{q}} \bigg( \int_x^{\infty} V(s)^{-p'} v(s)\,ds \bigg)^{\frac{1}{p'}} \\
& \hspace{-5.5cm} + \sup_{x \in (0,\infty)} \bigg( \int_x^{\infty} w(t)\,dt \bigg)^{\frac{1}{q}} \bigg( \int_0^x \bigg(\int_0^s a(y)u(y)\,dy\bigg)^{p'} V(s)^{-p'} v(s)\,ds \bigg)^{\frac{1}{p'}} \\
& \hspace{-5.5cm} + \sup_{x \in (0,\infty)} \bigg( \int_x^{\infty} \bigg( \int_x^t \frac{a(\tau)}{B(\tau)} u(\tau)\,d\tau \bigg)^q w(t)\,dt \bigg)^{\frac{1}{q}} \bigg( \int_0^x \bigg( \frac{B(s)}{V(s)}\bigg)^{p'} v(s)\,ds \bigg)^{\frac{1}{p'}} \\
& \hspace{-5.5cm} + \sup_{x \in (0,\infty)} \bigg( \int_x^{\infty} w(t)\,dt \bigg)^{\frac{1}{q}} \bigg( \int_0^x \bigg( \int_s^x \frac{a(\tau)}{B(\tau)} u(\tau)\,d\tau \bigg)^{p'} \bigg( \frac{B(s)}{V(s)}\bigg)^{p'} v(s)\,ds \bigg)^{\frac{1}{p'}} \\
& \hspace{-5.5cm} + \bigg( \int_0^{\infty} v(s)\,ds\bigg)^{-\frac{1}{p}}\, \bigg( \int_0^{\infty} \bigg(\int_0^x a(t)u(t)\,dt\bigg)^q w(x)\,dx \bigg)^{\frac{1}{q}};
\end{align*}

{\rm (ii)} If $q< p$, then
\begin{align*}
\sup_{f \in \mp^{+,\dn} (0,\infty)} \frac{\bigg( \int_0^{\infty} \bigg( \int_0^x  (P_{u,b} f)(t)  a(t)\,dt \bigg)^q w(x)\,dx \bigg)^{\frac{1}{q}}}{\bigg( \int_0^{\infty} f(s)^pv(s)\,ds\bigg)^{\frac{1}{p}}} & \\
& \hspace{-6cm} \approx \bigg( \int_0^{\infty} \bigg( \int_0^x \bigg(\int_0^t a(y)u(y)\,dy\bigg)^q w(t)\,dt \bigg)^{\frac{r}{p}} \, \bigg( \int_x^{\infty}  V(z)^{-p'} v(z)\,dz\bigg)^{\frac{r}{p'}} \, \bigg(\int_0^x a(y)u(y)\,dy\bigg)^q w(x)\,dx\bigg)^{\frac{1}{r}} \\
& \hspace{-5.5cm} + \bigg( \int_0^{\infty} \bigg( \int_x^{\infty} w(t)\,dt \bigg)^{\frac{r}{p}} \, \bigg( \int_0^x  \bigg(\int_0^z a(y)u(y)\,dy\bigg)^{p'} V(z)^{-p'} v(z)\,dz\bigg)^{\frac{r}{p'}} \, w(x)\,dx\bigg)^{\frac{1}{r}} \\
& \hspace{-5.5cm} + \bigg( \int_0^{\infty} \bigg( \int_x^{\infty} w(t)\,dt \bigg)^{\frac{r}{p}} \, \bigg( \int_0^x \bigg( \int_z^x \frac{a(\tau)}{B(\tau)} u(\tau)\,d\tau \bigg)^{p'} \bigg( \frac{B(z)}{V(z)}\bigg)^{p'} v(z)\,dz\bigg)^{\frac{r}{p'}} \, w(x)\,dx\bigg)^{\frac{1}{r}} \\
& \hspace{-5.5cm} + \bigg( \int_0^{\infty} \bigg( \int_x^{\infty} w(t) \bigg( \int_x^t \frac{a(\tau)}{B(\tau)} u(\tau)\,d\tau \bigg)^q \,dt \bigg)^{\frac{r}{q}}  \, \bigg( \int_0^x \bigg( \frac{B(s)}{V(s)}\bigg)^{p'} v(s)\,ds\bigg)^{\frac{r}{q'}} \, \bigg( \frac{B(x)}{V(x)}\bigg)^{p'} v(x) \,dx\bigg)^{\frac{1}{r}} \\
& \hspace{-5.5cm} + \bigg( \int_0^{\infty} v(s)\,ds\bigg)^{-\frac{1}{p}}\,\bigg( \int_0^{\infty} \bigg(\int_0^x a(y)u(y)\,dy\bigg)^q w(x)\,dx \bigg)^{\frac{1}{q}}.	
\end{align*}

\end{thm}
	
\begin{proof}
	By \cite[Theorem 3.1]{GogStep}, using Fubini's Theorem, we get that
	\begin{align*}
	\sup_{f \in \mp^{+,\dn} (0,\infty)} \frac{\bigg( \int_0^{\infty} \bigg( \int_0^x  (P_{u,b} f)(t)  a(t)\,dt \bigg)^q w(x)\,dx \bigg)^{\frac{1}{q}}}{\bigg( \int_0^{\infty} f(s)^pv(s)\,ds\bigg)^{\frac{1}{p}}} & \\
	& \hspace{-6cm} = \sup_{f \in \mp^{+,\dn} (0,\infty)} \frac{\bigg( \int_0^{\infty} \bigg( \int_0^x \bigg( \frac{u(t)}{B(t)} \int_0^t f(\tau) b(\tau) \,d \tau \bigg) a(t)\,dt \bigg)^q w(x)\,dx \bigg)^{\frac{1}{q}}}{\bigg( \int_0^{\infty} f(s)^pv(s)\,ds\bigg)^{\frac{1}{p}}} & \\
	& \hspace{-6cm} \approx \sup_{h \ge  0} \frac{\bigg( \int_0^{\infty} \bigg( \int_0^x \bigg( \frac{u(t)}{B(t)} \int_0^t \bigg( \int_{\tau}^{\infty} h(y)\,dy \bigg) b(\tau) \,d \tau \bigg) a(t) \, dt \bigg)^q w(x)\,dx \bigg)^{\frac{1}{q}}}{\bigg( \int_0^{\infty} h(s)^p  V(s)^p v(s)^{1-p}\,ds\bigg)^{\frac{1}{p}}} + \frac{\bigg( \int_0^{\infty} \bigg(\int_0^x a(t)u(t)\,dt\bigg)^q w(x)\,dx \bigg)^{\frac{1}{q}}}{\bigg( \int_0^{\infty} v(s)\,ds\bigg)^{\frac{1}{p}}} \\
	& \hspace{-6cm} \approx \sup_{h \ge  0} \frac{\bigg( \int_0^{\infty} \bigg( \int_0^x \bigg( \int_t^{\infty} h(y)\,dy \bigg) a(t)u(t)\, dt \bigg)^q w(x)\,dx \bigg)^{\frac{1}{q}}}{\bigg( \int_0^{\infty} h(s)^p  V(s)^p v(s)^{1-p}\,ds\bigg)^{\frac{1}{p}}} \\
	& \hspace{-5.5cm} + \sup_{h \ge  0} \frac{\bigg( \int_0^{\infty} \bigg( \int_0^x \bigg( \frac{u(t)}{B(t)} \int_0^t  h(y) B(y)\,dy \bigg) a(t) \, dt \bigg)^q w(x)\,dx \bigg)^{\frac{1}{q}}}{\bigg( \int_0^{\infty} h(s)^p  V(s)^p v(s)^{1-p}\,ds\bigg)^{\frac{1}{p}}} + \frac{\bigg( \int_0^{\infty}\bigg(\int_0^x a(t)u(t)\,dt\bigg)^q w(x)\,dx \bigg)^{\frac{1}{q}}}{\bigg( \int_0^{\infty} v(s)\,ds\bigg)^{\frac{1}{p}}} \\
	& \hspace{-6cm} \approx \sup_{h \ge  0} \frac{\bigg( \int_0^{\infty} \bigg( \int_x^{\infty} h(y)\,dy \bigg)^q \bigg(\int_0^x a(t)u(t)\,dt\bigg)^q w(x)\,dx \bigg)^{\frac{1}{q}}}{\bigg( \int_0^{\infty} h(s)^p  V(s)^p v(s)^{1-p}\,ds\bigg)^{\frac{1}{p}}} + \sup_{h \ge  0} \frac{\bigg( \int_0^{\infty} \bigg( \int_0^x  h(y) \bigg(\int_0^y a(t)u(t)\,dt\bigg)\,dy \bigg)^q w(x)\,dx \bigg)^{\frac{1}{q}}}{\bigg( \int_0^{\infty} h(s)^p  V(s)^p v(s)^{1-p}\,ds\bigg)^{\frac{1}{p}}} \\
    & \hspace{-5.5cm} + \sup_{h \ge  0} \frac{\bigg( \int_0^{\infty} \bigg( \int_0^x h(y) \bigg( \int_y^x \frac{a(t)}{B(t)} u(t)\, dt \bigg) \,dy \bigg)^q w(x)\,dx \bigg)^{\frac{1}{q}}}{\bigg( \int_0^{\infty} h(s)^p B(s)^{-p} V(s)^p v(s)^{1-p}\,ds\bigg)^{\frac{1}{p}}} + \frac{\bigg( \int_0^{\infty} \bigg(\int_0^x a(t)u(t)\,dt\bigg)^q w(x)\,dx \bigg)^{\frac{1}{q}}}{\bigg( \int_0^{\infty} v(s)\,ds\bigg)^{\frac{1}{p}}}.	
	\end{align*}

	{\rm (i)} Let $p \le q$. Using the characterizations of weighted Hardy-type inequalities (see, for instance, \cite[Section 1]{ok}), by \cite[Theorem 1.1]{Oinar}, we obtain that

	\begin{align*}
   	\sup_{f \in \mp^{+,\dn} (0,\infty)} \frac{\bigg( \int_0^{\infty} \bigg( \int_0^x  (P_{u,b} f)(t)  a(t)\,dt \bigg)^q w(x)\,dx \bigg)^{\frac{1}{q}}}{\bigg( \int_0^{\infty} f(s)^pv(s)\,ds\bigg)^{\frac{1}{p}}} & \\
    & \hspace{-6cm} \approx \sup_{x \in (0,\infty)} \bigg( \int_0^x \bigg(\int_0^t a(y)u(y)\,dy\bigg)^q w(t)\,dt \bigg)^{\frac{1}{q}} \bigg( \int_x^{\infty} V(s)^{-p'} v(s)\,ds \bigg)^{\frac{1}{p'}} \\
    & \hspace{-5.5cm} + \sup_{x \in (0,\infty)} \bigg( \int_x^{\infty} w(t)\,dt \bigg)^{\frac{1}{q}} \bigg( \int_0^x \bigg(\int_0^s a(y)u(y)\,dy\bigg)^{p'} V(s)^{-p'} v(s)\,ds \bigg)^{\frac{1}{p'}} \\
    & \hspace{-5.5cm} + \sup_{x \in (0,\infty)} \bigg( \int_x^{\infty} \bigg( \int_x^t \frac{a(\tau)}{B(\tau)} u(\tau)\,d\tau \bigg)^q w(t)\,dt \bigg)^{\frac{1}{q}} \bigg( \int_0^x \bigg( \frac{B(s)}{V(s)}\bigg)^{p'} v(s)\,ds \bigg)^{\frac{1}{p'}} \\
    & \hspace{-5.5cm} + \sup_{x \in (0,\infty)} \bigg( \int_x^{\infty} w(t)\,dt \bigg)^{\frac{1}{q}} \bigg( \int_0^x \bigg( \int_s^x \frac{a(\tau)}{B(\tau)} u(\tau)\,d\tau \bigg)^{p'} \bigg( \frac{B(s)}{V(s)}\bigg)^{p'} v(s)\,ds \bigg)^{\frac{1}{p'}} \\
    & \hspace{-5.5cm} + \bigg( \int_0^{\infty} v(s)\,ds\bigg)^{-\frac{1}{p}}\, \bigg( \int_0^{\infty} \bigg(\int_0^x a(t)u(t)\,dt\bigg)^q w(x)\,dx \bigg)^{\frac{1}{q}}.	
    \end{align*}

	{\rm (ii)} Let now $q < p$. Using the characterizations of weighted Hardy-type inequalities (see, for instance, \cite[Section 1]{ok}), by \cite[Theorem 1.2]{Oinar}, we obtain that
	\begin{align*}
    \sup_{f \in \mp^{+,\dn} (0,\infty)} \frac{\bigg( \int_0^{\infty} \bigg( \int_0^x  (P_{u,b} f)(t)  a(t)\,dt \bigg)^q w(x)\,dx \bigg)^{\frac{1}{q}}}{\bigg( \int_0^{\infty} f(s)^pv(s)\,ds\bigg)^{\frac{1}{p}}} & \\
    & \hspace{-6cm} \approx \bigg( \int_0^{\infty} \bigg( \int_0^x \bigg(\int_0^t a(y)u(y)\,dy\bigg)^q w(t)\,dt \bigg)^{\frac{r}{p}} \, \bigg( \int_x^{\infty}  V(z)^{-p'} v(z)\,dz\bigg)^{\frac{r}{p'}} \, \bigg(\int_0^x a(y)u(y)\,dy\bigg)^q w(x)\,dx\bigg)^{\frac{1}{r}} \\
    & \hspace{-5.5cm} + \bigg( \int_0^{\infty} \bigg( \int_x^{\infty} w(t)\,dt \bigg)^{\frac{r}{p}} \, \bigg( \int_0^x  \bigg(\int_0^z a(y)u(y)\,dy\bigg)^{p'} V(z)^{-p'} v(z)\,dz\bigg)^{\frac{r}{p'}} \, w(x)\,dx\bigg)^{\frac{1}{r}} \\
    & \hspace{-5.5cm} + \bigg( \int_0^{\infty} \bigg( \int_x^{\infty} w(t)\,dt \bigg)^{\frac{r}{p}} \, \bigg( \int_0^x \bigg( \int_z^x \frac{a(\tau)}{B(\tau)} u(\tau)\,d\tau \bigg)^{p'} \bigg( \frac{B(z)}{V(z)}\bigg)^{p'} v(z)\,dz\bigg)^{\frac{r}{p'}} \, w(x)\,dx\bigg)^{\frac{1}{r}} \\
    & \hspace{-5.5cm} + \bigg( \int_0^{\infty} \bigg( \int_x^{\infty} w(t) \bigg( \int_x^t \frac{a(\tau)}{B(\tau)} u(\tau)\,d\tau \bigg)^q \,dt \bigg)^{\frac{r}{q}}  \, \bigg( \int_0^x \bigg( \frac{B(s)}{V(s)}\bigg)^{p'} v(s)\,ds\bigg)^{\frac{r}{q'}} \, \bigg( \frac{B(x)}{V(x)}\bigg)^{p'} v(x) \,dx\bigg)^{\frac{1}{r}} \\
    & \hspace{-5.5cm} + \bigg( \int_0^{\infty} v(s)\,ds\bigg)^{-\frac{1}{p}}\,\bigg( \int_0^{\infty} \bigg(\int_0^x a(y)u(y)\,dy\bigg)^q w(x)\,dx \bigg)^{\frac{1}{q}}.	
    \end{align*}
	
	The proof is completed.
\end{proof}

\

\section{The boundedness of $T_{u,b}$ from $L^p(v)$ into $\ces_{q}(w,a)$ on the cone of monotone non-increasing functions}\label{T}

\

In this section we combine the results from previous two sections to present the characterization of the boundedness of $T_{u,b}$ from $L^p(v)$ into $\ces_q(w,a)$ on the cone of monotone non-increasing functions.	
\begin{thm}\label{thm.T}
	Let $1 < p,\, q < \infty$ and $b \in \W\I$ be such that $b(t) > 0$ for a.e. $t\in (0,\infty)$. Assume that $u \in \W\I \cap C\I$ and $a,\,v,\,w \in \W\I$. Moreover, assume that condition \eqref{add.cond.} holds.

	{\rm (i)} If $p \le q$, then
	\begin{align*}
	\sup_{f \in \mp^{+,\dn} (0,\infty)} \frac{\bigg( \int_0^{\infty} \bigg( \int_0^x (T_{u,b} f) (t) a(t)\,dt \bigg)^q w(x)\,dx \bigg)^{\frac{1}{q}}}{\bigg( \int_0^{\infty} f(s)^pv(s)\,ds\bigg)^{\frac{1}{p}}} & \\
	& \hspace{-6cm} \approx \, \sup_{t \in (0,\infty)} \bigg( \int_0^t V(x)^{p'} v(x) \, \bigg( \int_x^t \bigg( \sup_{s \le \tau}u(\tau) V(\tau)^{-2}\bigg) a(s)\,ds \bigg)^{p'} \, dx \bigg)^{\frac{1}{p'}} \, \bigg( \int_t^{\infty} w(y) \,dy \bigg)^{\frac{1}{q}} \notag \\
	& \hspace{-5.5cm} + \, \sup_{t \in (0,\infty)} \bigg( \int_0^t V(x)^{p'} v(x)  \, dx \bigg)^{\frac{1}{p'}} \, \bigg( \int_t^{\infty} \, \bigg( \int_t^y \bigg( \sup_{s \le \tau}u(\tau) V(\tau)^{-2}\bigg) a(s)\,ds \bigg)^{q} w(y) \,dy \bigg)^{\frac{1}{q}} \\
	& \hspace{-5.5cm} + \, \sup_{x \in (0,\infty)} \bigg( \int_{[x,\infty)} \, d \, \bigg( - \sup_{t \le \tau} u(\tau)^{p'} V(\tau)^{-2p'} \bigg( \int_0^{\tau} V(s)^{p'} v(s)\,ds \bigg) \bigg) \bigg)^{\frac{1}{p'}} \, \bigg( \int_0^x A(y)^q w(y) \,dy \bigg)^{\frac{1}{q}} \notag \\
	& \hspace{-5.5cm} + \, \sup_{x \in (0,\infty)} \bigg( \int_{(0,x]} \,  A(t)^{p'} \, d \, \bigg( - \sup_{t \le \tau} u(\tau)^{p'} V(\tau)^{-2p'} \bigg( \int_0^{\tau} V(s)^{p'} v(s)\,ds \bigg) \bigg) \bigg)^{\frac{1}{p'}} \, \bigg( \int_x^{\infty} w(y) \,dy \bigg)^{\frac{1}{q}} \notag \\
	& \hspace{-5.5cm} + \, \bigg( \int_0^{\infty} A(y)^q w(y) \,dy \bigg)^{\frac{1}{q}} \, \lim_{t \rightarrow \infty} \bigg(\sup_{t \le \tau} u(\tau) V(\tau)^{-2} \bigg( \int_0^{\tau} V(s)^{p'} v(s)\,ds \bigg)^{\frac{1}{p'}}\bigg) \\
	& \hspace{-5.5cm} + \, \sup_{x \in (0,\infty)} \bigg( \int_0^x \bigg(\int_0^t a(y)\bar{u}(y)\,dy\bigg)^q w(t)\,dt \bigg)^{\frac{1}{q}} \bigg( \int_x^{\infty} V(s)^{-p'} v(s)\,ds \bigg)^{\frac{1}{p'}} \\
	& \hspace{-5.5cm} + \, \sup_{x \in (0,\infty)} \bigg( \int_x^{\infty} w(t)\,dt \bigg)^{\frac{1}{q}} \bigg( \int_0^x \bigg(\int_0^s a(y)\bar{u}(y)\,dy\bigg)^{p'} V(s)^{-p'} v(s)\,ds \bigg)^{\frac{1}{p'}} \\
	& \hspace{-5.5cm} + \, \sup_{x \in (0,\infty)} \bigg( \int_x^{\infty} \bigg( \int_x^t \frac{a(\tau)}{B(\tau)} \bar{u}(\tau)\,d\tau \bigg)^q w(t)\,dt \bigg)^{\frac{1}{q}} \bigg( \int_0^x \bigg( \frac{B(s)}{V(s)}\bigg)^{p'} v(s)\,ds \bigg)^{\frac{1}{p'}} \\
	& \hspace{-5.5cm} + \, \sup_{x \in (0,\infty)} \bigg( \int_x^{\infty} w(t)\,dt \bigg)^{\frac{1}{q}} \bigg( \int_0^x \bigg( \int_s^x \frac{a(\tau)}{B(\tau)} \bar{u}(\tau)\,d\tau \bigg)^{p'} \bigg( \frac{B(s)}{V(s)}\bigg)^{p'} v(s)\,ds \bigg)^{\frac{1}{p'}} \\
	& \hspace{-5.5cm} + \, \bigg( \int_0^{\infty} v(s)\,ds\bigg)^{-\frac{1}{p}}\, \bigg( \int_0^{\infty} \bigg(\int_0^x a(t)\bar{u}(t)\,dt\bigg)^q w(x)\,dx \bigg)^{\frac{1}{q}};
	\end{align*}
	
	{\rm (ii)} If $q < p$, then
	\begin{align*}
	\sup_{f \in \mp^{+,\dn} (0,\infty)} \frac{\bigg( \int_0^{\infty} \bigg( \int_0^x (T_{u,b} f) (t) a(t)\,dt \bigg)^q w(x)\,dx \bigg)^{\frac{1}{q}}}{\bigg( \int_0^{\infty} f(s)^p v(s)\,ds\bigg)^{\frac{1}{p}}} & \\
	& \hspace{-6cm} \approx \, \bigg( \int_0^{\infty}  \bigg( \int_0^t V(x)^{p'} v(x) \,dx \bigg)^{\frac{r}{q'}} V(t)^{p'} v(t) \, \bigg( \int_t^{\infty} \bigg( \int_t^z \bigg( \sup_{s \le y}u(y)V(y)^{-2}\bigg) a(s)\,ds \bigg)^{q} w(z)\, dz \bigg)^{\frac{r}{q}} \, dt \bigg)^{\frac{1}{r}} \notag \\
	& \hspace{-5.5cm} + \, \bigg( \int_0^{\infty} \bigg( \int_0^t V(x)^{p'} v(x) \bigg( \int_x^t \bigg( \sup_{s \le y}u(y)V(y)^{-2}\bigg) a(s)\,ds \bigg)^{p'} \, dx \bigg)^{\frac{r}{p'}} \bigg( \int_z^{\infty} w(s) \,ds \bigg)^{\frac{r}{p}} w(t)\,dt \bigg)^{\frac{1}{r}} \\
	& \hspace{-5.5cm} + \, \bigg( \int_0^{\infty} \bigg( \int_{[x,\infty)} \, d \, \bigg( - \bigg(\sup_{t \le \tau} u(\tau)^{p'}V(\tau)^{-2p'} \, \bigg( \int_0^{\tau} V(s)^{p'} v(s)\,ds \bigg) \bigg) \bigg) \bigg)^{\frac{r}{p'}} \, \bigg( \int_0^x A(y)^q w(y) \,dy \bigg)^{\frac{r}{p}} x^q w(x) \,dx\bigg)^{\frac{1}{r}} \notag \\
	& \hspace{-5.5cm} + \, \bigg( \int_0^{\infty} \bigg( \int_{(0,x]} \,  A(t)^{p'} \, d \, \bigg( - \bigg(\sup_{t \le \tau} u(\tau)^{p'}V(\tau)^{-2p'} \, \bigg( \int_0^{\tau} V(s)^{p'} v(s)\,ds \bigg) \bigg) \bigg)  \bigg)^{\frac{r}{p'}} \, \bigg( \int_x^{\infty} w(y) \,dy \bigg)^{\frac{r}{p}} w(x)\,dx \bigg)^{\frac{1}{r}} \notag \\
	& \hspace{-5.5cm} + \, \bigg( \int_0^{\infty} A(y)^q w(y) \,dy \bigg)^{\frac{1}{q}} \, \lim_{t \rightarrow \infty} \bigg(\sup_{t \le \tau} u(\tau)V(\tau)^{-2} \bigg( \int_0^{\tau} V(s)^{p'} v(s)\,ds \bigg)^{\frac{1}{p'}}\bigg) \\
	& \hspace{-5.5cm} + \, \bigg( \int_0^{\infty} \bigg( \int_0^x \bigg(\int_0^t a(y)\bar{u}(y)\,dy\bigg)^q w(t)\,dt \bigg)^{\frac{r}{p}} \, \bigg( \int_x^{\infty}  V(z)^{-p'} v(z)\,dz\bigg)^{\frac{r}{p'}} \, \bigg(\int_0^x a(y)u(y)\,dy\bigg)^q w(x)\,dx\bigg)^{\frac{1}{r}} \\
	& \hspace{-5.5cm} + \, \bigg( \int_0^{\infty} \bigg( \int_x^{\infty} w(t)\,dt \bigg)^{\frac{r}{p}} \, \bigg( \int_0^x  \bigg(\int_0^z a(y)\bar{u}(y)\,dy\bigg)^{p'} V(z)^{-p'} v(z)\,dz\bigg)^{\frac{r}{p'}} \, w(x)\,dx\bigg)^{\frac{1}{r}} \\
	& \hspace{-5.5cm} + \, \bigg( \int_0^{\infty} \bigg( \int_x^{\infty} w(t)\,dt \bigg)^{\frac{r}{p}} \, \bigg( \int_0^x \bigg( \int_z^x \frac{a(\tau)}{B(\tau)} \bar{u}(\tau)\,d\tau \bigg)^{p'} \bigg( \frac{B(z)}{V(z)}\bigg)^{p'} v(z)\,dz\bigg)^{\frac{r}{p'}} \, w(x)\,dx\bigg)^{\frac{1}{r}} \\
	& \hspace{-5.5cm} + \, \bigg( \int_0^{\infty} \bigg( \int_x^{\infty} w(t) \bigg( \int_x^t \frac{a(\tau)}{B(\tau)} \bar{u}(\tau)\,d\tau \bigg)^q \,dt \bigg)^{\frac{r}{q}}  \, \bigg( \int_0^x \bigg( \frac{B(s)}{V(s)}\bigg)^{p'} v(s)\,ds\bigg)^{\frac{r}{q'}} \, \bigg( \frac{B(x)}{V(x)}\bigg)^{p'} v(x) \,dx\bigg)^{\frac{1}{r}} \\
	& \hspace{-5.5cm} + \, \bigg( \int_0^{\infty} v(s)\,ds\bigg)^{-\frac{1}{p}}\,\bigg( \int_0^{\infty} \bigg(\int_0^x a(y)\bar{u}(y)\,dy\bigg)^q w(x)\,dx \bigg)^{\frac{1}{q}}.
	\end{align*}   
	
\end{thm}

\begin{proof}
By \eqref{Split}, we have that
\begin{align*}
\sup_{f \in \mp^{+,\dn} (0,\infty)} \frac{\bigg( \int_0^{\infty} \bigg( \int_0^x (T_{u,b} f) (t) a(t)\,dt \bigg)^q w(x)\,dx \bigg)^{\frac{1}{q}}}{\bigg( \int_0^{\infty} f(s)^p v(s)\,ds\bigg)^{\frac{1}{p}}} \approx & \sup_{f \in \mp^{+,\dn} (0,\infty)} \frac{\bigg( \int_0^{\infty} \bigg( \int_0^x (R_u f) (t) a(t)\,dt \bigg)^q w(x)\,dx \bigg)^{\frac{1}{q}}}{\bigg( \int_0^{\infty} f(s)^pv(s)\,ds\bigg)^{\frac{1}{p}}} \\
& + \sup_{f \in \mp^{+,\dn} (0,\infty)} \frac{\bigg( \int_0^{\infty} \bigg( \int_0^x  (P_{\bar{u},b} f)(t)  a(t)\,dt \bigg)^q w(x)\,dx \bigg)^{\frac{1}{q}}}{\bigg( \int_0^{\infty} f(s)^pv(s)\,ds\bigg)^{\frac{1}{p}}}.
\end{align*}	
It remains to apply Theorems \ref{thm.R} and \ref{aux.thm.3}.
\end{proof}
\

\section{The boundedness of $M_{\gamma}$ from $\Lambda^p(v)$ into $\Gamma^q(w)$}\label{Appl.}

\

Suppose that $f$ is a measurable a.e. finite function on ${\mathbb
	R}^n$. Then its non-increasing rearrangement $f^*$ is given by
$$
f^* (t) = \inf \{\lambda > 0: |\{x \in {\mathbb R}^n:\, |f(x)| >
\lambda \}| \le t\}, \quad t \in (0,\infty),
$$
and let $f^{**}$ denotes the Hardy-Littlewood maximal function of
$f^*$, i.e.
$$
f^{**}(t) : = \frac{1}{t} \int_0^t f^* (\tau)\,d\tau, \quad t > 0.
$$
Quite many familiar function spaces can be defined by using the
non-increasing rearrangement of a function. One of the most
important classes of such spaces are the so-called classical Lorentz
spaces.

Let $p \in (0,\infty)$ and $w \in {\mathcal W}(0,\infty)$. Then the
classical Lorentz spaces $\Lambda^p (w)$ and $\Gamma^p (w)$ consist
of all measurable functions $f$ on $\rn$ for which
$\|f\|_{\Lambda^p(w)} : = \|f^*\|_{p,w,(0,\infty)} < \infty$ and
$\|f\|_{\Gamma^p(w)} : = \|f^{**}\|_{p,w,(0,\infty)} < \infty$,
respectively.
For more information about the Lorentz $\Lambda$ and $\Gamma$ spaces
see e.g. \cite{cpss} and the references therein.

The fractional maximal operator, $M_{\gamma}$, $\gamma \in (0,n)$,
is defined at a locally integrable function $f$ on $\rn$ by
$$
(M_{\gamma} f) (x) := \sup_{Q \ni x} |Q|^{ \gamma / n - 1} \int_{Q}
|f(y)|\,dy,\quad x \in \rn.
$$
It was shown in \cite[Theorem 1.1]{ckop} that
\begin{equation}\label{frac.max op.eq.1.}
(M_{\gamma}f)^* (t) \ls \sup_{\tau > t} \tau^{\gamma / n - 1}
\int_0^{\tau} f^*(y)\,dy \ls  (M_{\gamma} \tilde{f})^* (t)
\end{equation}
for every locally integrable function $f$ on $\rn$ and $t \in \I$, where $\tilde{f}
(x) : = f^* (\omega_n |x|^n)$ and $\omega_n$ is the volume of the unit ball in $\rn$. 

The characterization of the boundedness of $M_{\gamma}$ between classical Lorentz
spaces $\La^p(v)$ and $\La^q (w)$ was obtained in \cite{ckop} for the particular case when $1 < p \le
q <\infty$ and in \cite[Theorem 2.10]{o} in the case of more general operators and for extended range of $p$ and $q$ (For the characteriation of the boundedness of more general fractional maximal functions between $\La^p(v)$ and $\La^q (w)$, see \cite{musbil}, and the references therein).  

As an application of obtained results, we calculate the norm of the fractional maximal function $M_{\gamma}$ from $\Lambda^p(v)$ into $\Gamma^q(w)$.
\begin{thm}
	Let $1 < p,\, q < \infty$ and $0 < \gamma< n$. Assume that $v,\,w \in \W\I$.

	{\rm (i)} If $p \le q$, then
	\begin{align*}
	\|M_{\gamma}\|_{\Lambda^p(v) \rw \Gamma^q(w)} & \\
	& \hspace{-2cm} \approx \, \sup_{t \in (0,\infty)} \bigg( \int_0^t V(x)^{p'} v(x) \, \bigg( \int_x^t \bigg( \sup_{s \le \tau} \tau^{\frac{\gamma}{n}} V(\tau)^{-2}\bigg) \,ds \bigg)^{p'} \, dx \bigg)^{\frac{1}{p'}} \, \bigg( \int_t^{\infty} y^{-q} w(y) \,dy \bigg)^{\frac{1}{q}} \notag \\
	& \hspace{-1.5cm} + \sup_{t \in (0,\infty)} \bigg( \int_0^t V(x)^{p'} v(x)  \, dx \bigg)^{\frac{1}{p'}} \, \bigg( \int_t^{\infty} \, \bigg( \int_t^y \bigg( \sup_{s \le \tau}\tau^{\frac{\gamma}{n}} V(\tau)^{-2}\bigg) \,ds \bigg)^{q} y^{-q} w(y) \,dy \bigg)^{\frac{1}{q}} \\
	& \hspace{-1.5cm} + \, \sup_{x \in (0,\infty)} \bigg( \int_{[x,\infty)} \, d \, \bigg( - \sup_{t \le \tau} \tau^{\frac{\gamma}{n}p'} V(\tau)^{-2p'} \bigg( \int_0^{\tau} V(s)^{p'} v(s)\,ds \bigg) \bigg) \bigg)^{\frac{1}{p'}} \, \bigg( \int_0^x w(y) \,dy \bigg)^{\frac{1}{q}} \notag \\
	& \hspace{-1.5cm} + \sup_{x \in (0,\infty)} \bigg( \int_{(0,x]} \,  t^{p'} \, d \, \bigg( - \sup_{t \le \tau} \tau^{\frac{\gamma}{n}p'} V(\tau)^{-2p'} \bigg( \int_0^{\tau} V(s)^{p'} v(s)\,ds \bigg) \bigg) \bigg)^{\frac{1}{p'}} \, \bigg( \int_x^{\infty} y^{-q} w(y) \,dy \bigg)^{\frac{1}{q}} \notag \\
	& \hspace{-1.5cm} + \bigg( \int_0^{\infty} w(y) \,dy \bigg)^{\frac{1}{q}} \, \lim_{t \rightarrow \infty} \bigg(\sup_{t \le \tau} \tau^{\frac{\gamma}{n}} V(\tau)^{-2} \bigg( \int_0^{\tau} V(s)^{p'} v(s)\,ds \bigg)^{\frac{1}{p'}}\bigg) \\
	& \hspace{-1.5cm} + \sup_{x \in (0,\infty)} \bigg( \int_0^x t^{\frac{\gamma}{n} q} w(t)\,dt \bigg)^{\frac{1}{q}} \bigg( \int_x^{\infty} V(s)^{-p'} v(s)\,ds \bigg)^{\frac{1}{p'}} \\
	& \hspace{-1.5cm} + \sup_{x \in (0,\infty)} \bigg( \int_x^{\infty} y^{-q} w(y)\,dy \bigg)^{\frac{1}{q}} \bigg( \int_0^x t^{(\frac{\gamma}{n}+1)p'} V(t)^{-p'} v(t)\,dt \bigg)^{\frac{1}{p'}} \\
	& \hspace{-1.5cm} + \sup_{x \in (0,\infty)} \bigg( \int_x^{\infty} \bigg( t^{\frac{\gamma}{n}} - x^{\frac{\gamma}{n}}\bigg)^q t^{-q} w(t)\,dt \bigg)^{\frac{1}{q}} \bigg( \int_0^x s^{p'} V(s)^{-p'} v(s)\,ds \bigg)^{\frac{1}{p'}} \\
	& \hspace{-1.5cm} + \sup_{x \in (0,\infty)} \bigg( \int_x^{\infty} t^{-q} w(t)\,dt \bigg)^{\frac{1}{q}} \bigg( \int_0^x \bigg( x^{\frac{\gamma}{n}} - s^{\frac{\gamma}{n}} \bigg)^{p'} s^{p'} V(s)^{-p'} v(s)\,ds \bigg)^{\frac{1}{p'}} \\
	& \hspace{-1.5cm} + \bigg( \int_0^{\infty} v(s)\,ds\bigg)^{-\frac{1}{p}} \, \bigg( \int_0^{\infty} x^{\frac{\gamma}{n} q} w(x)\,dx \bigg)^{\frac{1}{q}};
	\end{align*}
	
	{\rm (ii)} If $q < p$, then
	\begin{align*}
	\|M_{\gamma}\|_{\Lambda^p(v) \rw \Gamma^q(w)} & \\
	& \hspace{-2cm} \approx \, \, \bigg( \int_0^{\infty}  \bigg( \int_0^t V(x)^{p'} v(x) \,dx \bigg)^{\frac{r}{q'}} V(t)^{p'} v(t) \, \bigg( \int_t^{\infty} \bigg( \int_t^z \bigg( \sup_{s \le y}y^{\frac{\gamma}{n}}V(y)^{-2}\bigg) \,ds \bigg)^{q} z^{-q} w(z)\, dz \bigg)^{\frac{r}{q}} \, dt \bigg)^{\frac{1}{r}} \notag \\
	& \hspace{-1.5cm} + \bigg( \int_0^{\infty} \bigg( \int_0^t V(x)^{p'} v(x) \bigg( \int_x^t \bigg( \sup_{s \le y}y^{\frac{\gamma}{n}}V(y)^{-2}\bigg) \,ds \bigg)^{p'} \, dx \bigg)^{\frac{r}{p'}} \bigg( \int_z^{\infty} s^{-q} w(s) \,ds \bigg)^{\frac{r}{p}} t^{-q} w(t)\,dt \bigg)^{\frac{1}{r}} \\
	& \hspace{-1.5cm} + \, \bigg( \int_0^{\infty} \bigg( \int_{[x,\infty)} \, d \, \bigg( - \bigg(\sup_{t \le \tau} \tau^{\frac{\gamma}{n}p'} V(\tau)^{-2p'} \, \bigg( \int_0^{\tau} V(s)^{p'} v(s)\,ds \bigg) \bigg) \bigg) \bigg)^{\frac{r}{p'}}\, \bigg( \int_0^x w(y) \,dy \bigg)^{\frac{r}{p}} w(x) \,dx\bigg)^{\frac{1}{r}} \notag \\
	& \hspace{-1.5cm} + \bigg( \int_0^{\infty} \bigg( \int_{(0,x]} \,  t^{p'} \, d \, \bigg( - \bigg(\sup_{t \le \tau} \tau^{\frac{\gamma}{n}p'} V(\tau)^{-2p'} \, \bigg( \int_0^{\tau} V(s)^{p'} v(s)\,ds \bigg) \bigg) \bigg)  \bigg)^{\frac{r}{p'}} \, \bigg( \int_x^{\infty} y^{-q} w(y) \,dy \bigg)^{\frac{r}{p}} x^{-q} w(x)\,dx \bigg)^{\frac{1}{r}} \notag \\
	& \hspace{-1.5cm} + \bigg( \int_0^{\infty} w(y) \,dy \bigg)^{\frac{1}{q}} \, \lim_{t \rightarrow \infty} \bigg(\sup_{t \le \tau} \tau^{\frac{\gamma}{n}p'} V(\tau)^{-2} \bigg( \int_0^{\tau} V(s)^{p'} v(s)\,ds \bigg)^{\frac{1}{p'}}\bigg) \\
	& \hspace{-1.5cm} + \bigg( \int_0^{\infty} \bigg( \int_0^x t^{\frac{\gamma}{n} q} w(t)\,dt \bigg)^{\frac{r}{p}} \, \bigg( \int_x^{\infty}  V(z)^{-p'} v(z)\,dz\bigg)^{\frac{r}{p'}} \, x^{\frac{\gamma}{n} q} w(x)\,dx\bigg)^{\frac{1}{r}} \\
	& \hspace{-1.5cm} + \bigg( \int_0^{\infty} \bigg( \int_x^{\infty} t^{-q} w(t)\,dt \bigg)^{\frac{r}{p}} \, \bigg( \int_0^x  z^{(\frac{\gamma}{n} + 1) p'} V(z)^{-p'} v(z)\,dz\bigg)^{\frac{r}{p'}} \, x^{-q} w(x)\,dx\bigg)^{\frac{1}{r}} \\
	& \hspace{-1.5cm} + \bigg( \int_0^{\infty} \bigg( \int_x^{\infty} t^{-q} w(t)\,dt \bigg)^{\frac{r}{p}} \, \bigg( \int_0^x \bigg( x^{\frac{\gamma}{n}} - z^{\frac{\gamma}{n}} \bigg)^{p'} z^{p'} V(z)^{-p'} v(z)\,dz\bigg)^{\frac{r}{p'}} \, x^{-q} w(x)\,dx\bigg)^{\frac{1}{r}} \\
	& \hspace{-1.5cm} + \bigg( \int_0^{\infty} \bigg( \int_x^{\infty} \bigg( t^{\frac{\gamma}{n}} - x^{\frac{\gamma}{n}} \bigg)^q t^{-q} w(t) \,dt \bigg)^{\frac{r}{q}}  \, \bigg( \int_0^x s^{p'} V(s)^{-p'} v(s)\,ds\bigg)^{\frac{r}{q'}} \, x^{p'} V(x)^{-p'} v(x) \,dx\bigg)^{\frac{1}{r}} \\
	& \hspace{-1.5cm} + \bigg( \int_0^{\infty} v(s)\,ds\bigg)^{-\frac{1}{p}}\,\bigg( \int_0^{\infty} x^{\frac{\gamma}{n} q} w(x)\,dx \bigg)^{\frac{1}{q}}.
	\end{align*}   
	
\end{thm}

\begin{proof}
From inequalities \eqref{frac.max op.eq.1.}, we have that
$$
\|M_{\gamma}\|_{\Lambda^p(v) \rw \Gamma^q(w)} \approx \sup_{f \in \mp^{+,\dn} (0,\infty)} \frac{\bigg( \int_0^{\infty} \bigg( \int_0^x (T_{u,b} f) (t) \,dt \bigg)^q x^{-q} w(x)\,dx \bigg)^{\frac{1}{q}}}{\bigg( \int_0^{\infty} f(s)^p v(s)\,ds\bigg)^{\frac{1}{p}}} 
$$
with $u(\tau) = \tau^{\gamma / n}$ and $b \equiv 1$. Note that
\begin{equation*}
\sup_{0 < t < \infty} \frac{u(t)}{B(t)} \int_0^t
\frac{b(\tau)}{u(\tau)}\,d\tau < \infty
\end{equation*}
in this case. So, it remains to apply Theorem \ref{thm.T}.
\end{proof}

%


%
%
%
%
%
\end{document}